\documentclass[a4paper, 10pt, reqno]{amsart}

\usepackage[T1]{fontenc}
\usepackage[utf8]{inputenc}

\usepackage{amssymb}
\usepackage{amsmath}
\usepackage{amsfonts}
\usepackage{amsthm}
\usepackage{mathtools}

\usepackage{fixltx2e} 
\usepackage{isomath} 

\usepackage{comment}
\usepackage{soul}
\usepackage{url}

\usepackage{xcolor}
\usepackage{graphicx}
\usepackage{color}
\usepackage{tikz}
\usetikzlibrary{snakes}

\usepackage{url}%
\usepackage{hyperref}%
\hypersetup{
  colorlinks=false,
  linkbordercolor=gray,
  citebordercolor=gray,
  urlbordercolor=gray}%

\theoremstyle{plain}%
\newtheorem{theorem}{Theorem}[section]
\newtheorem{lemma}[theorem]{Lemma}

\newtheorem{corollary}[theorem]{Corollary}

\theoremstyle{definition}
\newtheorem{definition}[theorem]{Definition}
\newtheorem{example}[theorem]{Example}

\theoremstyle{remark}
\newtheorem{remark}[theorem]{Remark}

\title{Exact Hausdorff measures of Cantor sets}
\author{Malin Palö Forsström}
\thanks{University of Gothenburg and Chalmers University of Technology, Gothenburg, Sweden. E-mail: \texttt{palo@chalmers.se}}
\date{\today}

\usepackage[]{todonotes}

\definecolor{note_linecolor}{rgb}{0.5,0.5,0.5}
\definecolor{note_bordercolor}{RGB}{246, 104, 178}
\definecolor{note_backgroundcolor}{RGB}{250, 167, 200}



\begin{document}

\maketitle

\begin{abstract}
Cantor sets in \(\mathbb{R}\) are common examples of sets for which Hausdorff measures can be positive and finite. However, there exist Cantor sets for which no Hausdorff measure is supported and finite. The purpose of this paper is to try to resolve this problem by studying an extension of the Hausdorff measures \( \mu_h\) on \(\mathbb{R}\), allowing gauge functions to depend on the midpoint of the covering intervals instead of only on the diameter. As a main result, a theorem about the Hausdorff measure  of any regular enough Cantor set, with respect to a chosen gauge function, is obtained.
\end{abstract}

\section{Introduction.}

Felix Hausdorff, in his paper \textit{Dimension und äu\ss eres Ma\ss} from 1918, as translated by Sawhill in the book \textit{Classics on Fractals} \cite{classics}, made the following definition.
\begin{definition} \label{def:hausdorff_measure_def_by_Hausdorff}
  Let \( \mathcal{U} \) be a system of bounded sets \( U \) in a \( q \)-dimensional space having the property that one can cover any set \( A \) with an at most countable number of sets \( U \) from \( \mathcal{U} \) having arbitrarily small diameters \(  |U| \). Let \( h:\mathcal{U} \to [0,\infty) \) be a set function. Denote by 
  \begin{equation*}
    \mu_{\mathcal{U},h}^\delta (A) = \inf \sum h(U_n)
  \end{equation*}
   where the infinum runs over all countable subsets \( \{U_n\} \) of \( \mathcal{U} \) such that \( \cup U_n \) covers \( A \) and \( |U_n|<\delta \) for all \( n \). If \( \mathcal{U} \) is the set of Borel sets then \( \mu_{\mathcal{U},h} (A) = \lim_{\delta \to 0} \mu_{\mathcal{U},h}^\delta (A) \) is a measure. If \( h(U) \) is continuous or \( h(U)=h(\bar U ) \), then \( \mu_{\mathcal{U},h} \) is an outer measure.
\end{definition}
In \( \mathbb{R} \), with which we will be concerned, a common choice is to take \( \mathcal{U} \) to be the set of all intervals and to restrict the choice of the set function  \( h \) to interval functions depending on only the diameter of the sets on which it is applied. In this paper, we will use a definition somewhat closer to the original definition made by Hausdorff.
   
Let \( I(w,\delta) \) denote the interval with midpoint \( w \) and diameter \( \delta \).  If \( I = I(w,\delta)\) and there is no risk for confusion, we sometimes write \( h(I) \) instead of \( h(w, \delta) \). By a \( \delta \)-cover of a set \( E \) we will mean a collection of sets of diameter at most \( \delta \) whose union contain \( E \). Using these notations we can formulate the definition of Hausdorff measures on \( \mathbb{R} \)  which we will use. This definition differs from definitions previously used in the context of Cantor sets in \( \mathbb{R} \) in that the gauge function is allowed to depend not only on the diameter of the covering intervals, but also on their midpoints.
\begin{definition} \label{def:haussdorffmeasure}
  Let \( h : \mathbb{R} \times \mathbb{R}_+ \to \mathbb{R}_+ \) be a continuous function with \\\({ \lim_{\delta \to 0} h(w, \delta) = 0 }\) for all \( w\in \mathbb{R} \)  which is increasing as an interval function. Then the Hausdorff measure  of the set \( E \subseteq \mathbb{R} \) with respect to the gauge function \( h \) is defined by 
  \begin{equation*}
     \mu_h (E) = \lim_{\delta \to 0} \inf \left\{ \sum h(w_k,\delta_k), \; \textrm{ where } \{  I(w_k, \delta_k) \} \textrm{ is a } \delta\textrm{-covering of } E \right\}.
  \end{equation*}
\end{definition}
The function \( h \) in the definition above will be called the \textit{gauge function} associated with the measure \( \mu_h \) and \( \mu_h \) will be called the Hausdorff measure associated with the gauge function \( h \). Moreover, any function with the properties above will be called a gauge function.  The fact that the measure in Definition~\ref{def:haussdorffmeasure} is a well-defined outer measure follows from Definition~\ref{def:hausdorff_measure_def_by_Hausdorff}, even if the assumption on \( h \) being continuous is dropped. When the sets we want to measure are subsets of  \( \mathbb{R} \), we get a definition equivalent to definition~\ref{def:haussdorffmeasure} if we consider only coverings by intervals. Also, it can be shown (see e.g.~\cite{hausdorffmeasures}) that the resulting measure does not depend on whether the sets considered in the covering in the definition above are open or closed.

The reason for using the definition above instead of the more common definition requiring that \( h(w, \delta)  \) does not depend on \( w \), is that given this restriction makes it possible to find  Hausdorff measures which are finite and supported on a given Cantor set for Cantor sets for which this, using the more restrictive definition, is not possible. Also, assuming Lipschitz continuity of \( h(w, \delta) \) in the first argument, only small adaptions of the corresponding proofs for the case \( h(w,\delta) = h(\delta) \) (see e.q.~\cite{mattila}) are needed to show that results such as Frostman's lemma and common density bounds hold also in this setting\cite{malin}.

By a Cantor set in \( \mathbb{R} \) we mean a subset of \( \mathbb{R} \) which is compact, perfect and totally disconnected. Given the notation we will use throughout this paper, this definition translates as follows.
\begin{definition}
Let \[ C = \lim \limits_{n \to \infty} \bigcap \limits_{k=0}^n \bigcup_{j \in \{ 0,1 \}^k} I_j \] 
  where \( \mathcal{I} = \{ I_j \}_{j \in \{ 0,1 \}^k\, k = 0,1,2,\ldots }\) is a collection of nonempty closed intervals. Let \( j0 \) denote the concatenation of the two binary words \( j \) and \( 0 \), and \(j1\) denote the concatenation of the binary words \( j \) and \(1\). If for all \( I_j \in \mathcal{I} \),
  \begin{itemize}
    \item \( I_{j0} \cap I_{j1} = \emptyset \)
    \item \( I_{j0} \cup I_{j1} \subseteq I_j \) and
    \item \( I_{j0} \) and \( I_j \) have the same left endpoint and \( I_{j1} \) and \( I_j \) have the same right endpoint
  \end{itemize}
we say that \( C \) is a Cantor set, and write \( C \sim \{ I_j \} \).
\end{definition}

The intervals \(  I_j = I(w_j, \delta_j) \) appearing in the construction of a Cantor set \( C \) will be called the \emph{basic intervals} associated with the Cantor set.  Moreover, any interval whose left endpoint is the left endpoint of a basic interval and whose right endpoint is a right endpoint of a basic interval will be called a \emph{near~basic interval} associated with \( C \). We use \( G_j \) to denote the open interval \(  I_j \backslash \left( I_{j0} \cup I_{j1} \right) \), and say that \( G_j \) is a gap associated to the Cantor set \( C \sim \{ I_j \} \). When \( I = I(w, \delta) \) is an interval and \( a>0 \), we will write \( aI \) to denote the interval \( I(w, a\delta)\), i.e.\ we write
\[
aI = aI(w,\delta) = I(w,a\delta)
\]

The unique probability measure \( \nu \) satisfying \( \nu (I_{j0}) = \nu (I_{j1}) = \frac{1}{2} \, \nu(I_j) \)  for all binary words \( j \) is called the \textit{Cantor measure} associated with the Cantor set \( C \sim \{ I_j \} \). The fact that the Cantor measure is a well defined measure follows by Proposition 1.7 in \cite{falconer2}. More generally, a measure which non-trivial and finite and supported on a given set \( E \) is called a \textit{mass distribution} on~\( E \).

When \( j_1 \) and \( j_2 \) are two binary words, \( j_1j_2 \)  will denote their concatenation. Also, \( 0^m \) will be used throughout this text to denote the binary word which consists of \( m \) zeros. \( 1^m \)  is defined analogously. 

In this paper we will almost exclusively use binary words to enumerate the elements of the construction of a Cantor set. However, an alternative notation, which is simpler in some situations, is to use \( I_l^k \) to represent the \( l \)th interval in the \( k \)th construction step. If \( j \) is a binary word and we let \( j_{10} \) be the integer we get if converting \( j \) when considered as a binary number to base 10, we can convert between the two notations by \( I_j = I_{j_{10}}^{|j|} \). Similarly \( G_j = G_{j_{10}}^{|j|} \). We will only use this notation in examples~\ref{ex:1} to \ref{ex:3} and in the proof of Corollary~\ref{cor:the_cor}.

\section{Main results.}
Small adaptions of the standard methods for calculating Hausdorff measures of Cantor sets (see e.g.~\cite{mattila}, pp.~60-63) now yields the first of the two theorems below, which shows that many of the Hausdorff measures as defined in this paper are mass distributions on some Cantor sets. This fact motivates the use of this definition, as it extends the family of Cantor sets whose dimension we understand, in the sense of which gauge functions yield mass distributions on the sets through its associated Hausdorff measure. Similar results, but with less strict bounds, can easily be obtained when the ratio of \( h(I_j) \) and \( \nu(I_j) \) is bounded from above and below away from zero.

\begin{theorem} \label{thm:Lminthan1}
Let \( h \) be any gauge function and suppose there exists a constant \( D \)  such that \( D \cdot h(w,\delta)>h(w,2\delta) \) for all \( w \) and \( \delta \). Let \( C \sim \{ I_j \}\) be a Cantor set such that \( 2\max \bigl\{ |I_{j0}|,|I_{j1}| \bigr\} \leq |I_j| \)  and assume there exist two constants \( q \) and \( r \) such that   \(q \cdot \nu(I_j) \leq h(I_j) \leq r\cdot  \nu(I_j)  \), where \( \nu \) is the Cantor measure associated with \( C \). Then \( \mu_h \) is a mass distribution on \( C \).  Further, for any interval \( J \subseteq [0,1] \), \(  q/2D^2 \cdot \nu(J) \leq \mu_h (J \cap C) \leq r \cdot \nu (J) \)
\end{theorem}

While the previous theorem gives satisfactory information about the (local) dimension of a Cantor set (through gauge functions), it does not give specific information about the exact measure of any Cantor set. This is the main purpose of our main result, the theorem below, which, especially in the case \( r=q \), gives more explicit information about Hausdorff measures of Cantor sets, both globally and locally.

\begin{theorem} \label{thm:general-1-thm}
  Let \( J \subseteq [0,1] \) be any closed interval and let \( \varepsilon > 0 \) be a small positive number. Further let \( h \) be a gauge function and \( C \sim \{ I_j \}  \) be a Cantor set for which the following assumptions hold:
  \renewcommand{\labelenumi}{\textsc{\roman{enumi}}.}
  \begin{enumerate}
    \item for any fixed $w$ and small enough $\delta$ with \( I(w, \delta) \subseteq (1 + \varepsilon) \cdot J \) we have\footnote{We will throughout this paper use subindices to denote derivates, s.t. for example \( h_{11} (w, \delta) = \frac{\partial^2 h}{\partial w^2} (w, \delta) \).}
\begin{equation}
-h_{11}(w,\delta) + 4h_{22}(w,\delta) \leq 0
\label{eq:h1}
\end{equation}
      and
\begin{equation}
 h_{11}(w,\delta) +  4h_{12}(w,\delta) + 4h_{22}(w,\delta)\leq 0
\label{eq:h2}
\end{equation}

    \item for all long enough binary words \( j \) with \( I_j \subseteq (1 + \varepsilon) \cdot J \) and all \( m \in \mathbb{N} \)  the following inequality holds 
\begin{equation*}
\displaystyle \frac{1}{2^m} \leq \frac{|G_j \cup I_{j10^m}|}{|G_j \cup I_{j1}|} 
\end{equation*}

    \item there exist two positive numbers \( q \) and \( r \)  such that  for all long enough binary words \( j \)  with \( I_j \subseteq (1 + \varepsilon) \cdot J \) the following pair of inequalities hold
\begin{equation}
q \cdot  \nu (I_j) \leq h(I_j) \leq r \cdot \nu (I_j)
\label{eq:qrinequality}
\end{equation}

  \end{enumerate}
  Then 
\begin{equation}
\bigl( q - (r-q) \bigr) \cdot \nu(J)\leq \mu_h (J \cap C) \leq r \cdot \nu(J)
\label{eq:important}
\end{equation}
\end{theorem}

When the gauge function \( h \) only depends on the diameter of the covering intervals, i.e.\ when the gauge function is of the form \(  h(\delta) \), the first of the three assumptions above simplifies into  \( h \) being concave. This is a reasonable requirement since for the arguably most commonly studied Hausdorff measures in the context of Cantor sets; the Hausdorff measures associated to the gauge functions \( h(\delta) = \delta^\alpha \), the corresponding gauge function is concave for \( \alpha \in (0,1) \).

Also the second assumption simplifies in special cases. A well studied subset of the set of all Cantor sets in \( \mathbb{R} \) is the Cantor sets with so called decreasing gap sequences. We say that \( C \sim \{ I_j \} \) has a  decreasing gap sequence if \( |G_l^k| \leq |G_{l'}^{k'}|\) when \( 2^{k'}+{l'} < 2^k+l  \). When using this notation, by assumption we have
\[
|I_l^k| = \sum_ {n=0}^\infty \sum_{m=0}^{2^n} |G^{k+n}_{2^{n} l+m}| \qquad \text{and} \qquad |I_{l'}^{k'}| = \sum_ {n=0}^\infty \sum_{m=0}^{2^{n}} |G^{k'+n}_{2^{n} l'+m}| 
\]
and 
\[ 
2^{(k'+n)}+\left( 2^nl'+m\right) = 2^n(2^{k'}+l') + m < 2^n(2^k+l) + m = 2^{(k+n)}+\left( 2^nl+m\right)
\]
implying that \( |G^{k+n}_{2^{n} l+m}|  \leq  |G^{k'+n}_{2^{n} l'+m}|  \) for any fixed $m$ and $n$ when \( 2^{k'}+{l'} < 2^k+l  \). Comparing the two double sums above termwise, we see that this implies $|I_l^k| \leq |I_{l'}^{k'}| $, which means the interval sequence is decreasing in the same sense as the gap sequence is. This gives
\begin{equation*}
\begin{split}
|I_{j1}| =& \sum_{k \in \{ 0,1 \}^m} |I_{j1k}| + \sum_{j=0}^{m-1} \sum_{m\in \{ 0,1 \}^l} |G_{j1l}| 
\\
\leq& \sum_{k \in \{ 0,1 \}^m} |I_{j10^m}| + \sum_{j=0}^{m-1} \sum_{m\in \{ 0,1 \}^l} |G_{j}| 
\\[1ex]=&
\;\;\; 2^m|I_{j10^m}| + \left( 2^{m}-1 \right) |G_j|
\end{split}
\end{equation*}
Rearranging the terms above, we get
\[
\frac{1}{2^m} \leq \frac{|G_j| + |I_{j10^m}| }{|G_j|+|I_{j1}|}
\]
i.e.\ the second assumption of the theorem is satisfied for any Cantor set whose gap sequence is decreasing. This observation, together with the previous observation, yields the following corollary.

\begin{corollary} \label{cor:the_cor}
Let \( J \subseteq [0,1] \) be any closed interval and let \( \varepsilon > 0 \) be a small positive number. Further let \( h(\delta) \) be a concave gauge function and \( C \sim \{ I_j \}  \) be a Cantor set associated to a decreasing gap sequence for which  there exist two positive numbers \( q \) and \( r \)  such that  for all long enough binary words \( j \)  with \( I_j \subseteq (1 + \varepsilon) \cdot J \)  
\begin{equation*}
q \cdot  \nu (I_j) \leq h(I_j) \leq r \cdot \nu (I_j).
\end{equation*}
Then 
\begin{equation*}
\bigl( q - (r-q) \bigr) \cdot \nu(J)\leq \mu_h (J \cap C) \leq r \cdot \nu(J).
\end{equation*}
\end{corollary}

The rest of this paper is structured as follows. In the next section, we give a proof of our main result. In the last section, we use this result to calculate the exact Hausdorff measure of a family of Cantor sets, for which upper and lower estimates were given in \cite{pCantorset_measure}, and for which the measure (to the author's knowledge) was previously unknown.

\section{Proof of the main results.}

To be able to give a proof of  Theorem~\ref{thm:general-1-thm} and its subsequent corollary, we will need the following lemma. This lemma and its proof  use the notation   \( \rho \cdot_L I \) to denote  the leftmost \( \rho \)-proportion of the set \( I \), and analogously by \( \rho \cdot_R I \) the rightmost \( \rho \)-proportion of the set \( I \). Note that this implies that \( 1\cdot_L I = I \), \( 1 \cdot_R I = I \), \( 0 \cdot_L I = \emptyset \) and \( 0 \cdot_R I = \emptyset \).

\begin{lemma} \label{lem:staircase}
Let \( C\sim \{ I_j \} \) be a Cantor set. Let \( \{ G_j \} \) be the corresponding gap sequence and let \( \nu \) be the associated Cantor measure. Then the following claims are equivalent:
\renewcommand{\labelenumi}{(\roman{enumi})}
\begin{enumerate}
\item For all long enough binary words \( j \) and all \( \rho \in [0,1] \)
\begin{equation}
 \nu\bigl(\rho \cdot_L (G_j \cup I_{j1})\bigr) \leq \rho \cdot \nu (I_{j1}) 
\label{eq:firstequiv}
\end{equation} 
\item For all long enough binary words \( j \) and all \( m \in \mathbb{N} \) 
\begin{equation*}
\displaystyle \frac{1}{2^m} \leq \frac{|G_j \cup I_{j10^m}|}{|G_j \cup I_{j1}|} 
\end{equation*}  
\end{enumerate}
\end{lemma}

\begin{proof}[Proof of Lemma~\ref{lem:staircase}]

  We first show that (i) implies (ii). To this end, let \( j \) be any binary word which is long enough for (i) to hold and let \( m \in \mathbb{N} \). Set \( \rho = \smash{\frac{|G_j  \cup I_{j10^m}|}{|G_j  \cup I_{j1}|} } \) and note that this implies that \( {\rho \cdot_L (G_j \cup I_{j1}) = G_j \cup I_{j10^m}}\). Also
  \begin{equation} \label{eq:tired}
    \nu \, \bigl(\rho \cdot_L (G_j \cup I_{j1})\bigr) = \nu\, (G_j \cup I_{j10^m}) = \nu\, (I_{j10^m}) = \frac{1}{2^m} \nu\, (I_{j1})
  \end{equation}
  by the definition of the Cantor measure. Using this equation and applying (i) we get 
  \begin{equation*}
    \frac{1}{2^m} \, \nu\, (I_{j1}) \overset{(\ref{eq:tired})}{=} \nu \, \bigl(\rho \cdot_L (G_j \cup I_{j1})\bigr)  \overset{(i)}{\leq} \rho \cdot \nu(I_{j1}) = \frac{|G_j  \cup I_{j10^m}|}{|G_j  \cup I_{j1}|} \cdot \nu(I_{j1}) 
  \end{equation*}
  Dividing by \( \nu(I_{j1}) \) gives (ii).

 We will now show that the reverse implication holds, i.e.\ that (ii) implies (i), by showing that if (i) is false, then (ii) is also false. This part of the proof will rely heavily on the following notation. Namely, if \(i\) is a binary word, we will write \(\rho_{j1i}\) for the unique number in \( [0,1] \) such that
\[
\rho_{j1i} \cdot_L (G_j \cup I_{j1})  = G_j \cup [w_{j1}-\frac{\delta_{j1}}{2}, w_{j1i}+\frac{\delta_{j1i}}{2}]
\]
Here, as 
\[ I_{j1} = I(w_{j1}, \delta_{j1}) = [w_{j1}-\delta_{j1}/2, w_{j1}+\delta_{j1}/2] \] 
and 
\[ I_{j1i} = I(w_{j1i}, \delta_{j1i}) = [w_{j1i}-\delta_{j1i}/2, w_{j1i}+\delta_{j1i}/2], \]
\( w_{j1}-\delta_{j1}/2 \) is the left endpoint of \( I_{j1} \) and \( {w_{j1i}+\delta_{j1i}/2} \) is the right endpoint of \( I_{j1i} \). Although we will not use it below, we have that 
\[
\rho_{j1i} = \frac{\left| G_j \cup [w_{j1}-\frac{\delta_{j1}}{2}, w_{j1i}+\frac{\delta_{j1i}}{2}] \right|}{|G_j \cup I_{j1}|}
\]
Note that with this notation, for any binary word $i$ we have $\rho_{j1i1} = \rho_{j1i}$.

Suppose now that (i) is false. Then there is a binary word \( j \) and a number \( \rho \in [0,1] \) such that
\begin{equation}
\nu \left(\rho \cdot_L (G_j \cup I_{j1}) \right) > \rho \cdot \nu(I_{j1}).
\label{eq:noti}
\end{equation}
\begin{figure}[h]
\centering
\begin{tikzpicture}

\draw[gray] (0.0000,0.0000) -- (1.6000,0.0000);
\draw[gray] (1.6000,0.0000) -- (1.6011,0.0195);
\draw[gray] (1.6011,0.0195) -- (1.6017,0.0195);
\draw[gray] (1.6017,0.0195) -- (1.6031,0.0391);
\draw[gray] (1.6031,0.0391) -- (1.6067,0.0391);
\draw[gray] (1.6067,0.0391) -- (1.6076,0.0586);
\draw[gray] (1.6076,0.0586) -- (1.6079,0.0586);
\draw[gray] (1.6079,0.0586) -- (1.6089,0.0781);
\draw[gray] (1.6089,0.0781) -- (1.6152,0.0781);
\draw[gray] (1.6152,0.0781) -- (1.6161,0.0977);
\draw[gray] (1.6161,0.0977) -- (1.6165,0.0977);
\draw[gray] (1.6165,0.0977) -- (1.6172,0.1172);
\draw[gray] (1.6172,0.1172) -- (1.6181,0.1172);
\draw[gray] (1.6181,0.1172) -- (1.6186,0.1367);
\draw[gray] (1.6186,0.1367) -- (1.6194,0.1367);
\draw[gray] (1.6194,0.1367) -- (1.6199,0.1563);
\draw[gray] (1.6199,0.1563) -- (1.6333,0.1563);
\draw[gray] (1.6333,0.1563) -- (1.6340,0.1758);
\draw[gray] (1.6340,0.1758) -- (1.6352,0.1758);
\draw[gray] (1.6352,0.1758) -- (1.6358,0.1953);
\draw[gray] (1.6358,0.1953) -- (1.6369,0.1953);
\draw[gray] (1.6369,0.1953) -- (1.6384,0.2148);
\draw[gray] (1.6384,0.2148) -- (1.6391,0.2148);
\draw[gray] (1.6391,0.2148) -- (1.6400,0.2344);
\draw[gray] (1.6400,0.2344) -- (1.6470,0.2344);
\draw[gray] (1.6470,0.2344) -- (1.6477,0.2539);
\draw[gray] (1.6477,0.2539) -- (1.6485,0.2539);
\draw[gray] (1.6485,0.2539) -- (1.6492,0.2734);
\draw[gray] (1.6492,0.2734) -- (1.6511,0.2734);
\draw[gray] (1.6511,0.2734) -- (1.6527,0.2930);
\draw[gray] (1.6527,0.2930) -- (1.6529,0.2930);
\draw[gray] (1.6529,0.2930) -- (1.6544,0.3125);
\draw[gray] (1.6544,0.3125) -- (1.7440,0.3125);
\draw[gray] (1.7440,0.3125) -- (1.7447,0.3320);
\draw[gray] (1.7447,0.3320) -- (1.7461,0.3320);
\draw[gray] (1.7461,0.3320) -- (1.7470,0.3516);
\draw[gray] (1.7470,0.3516) -- (1.7501,0.3516);
\draw[gray] (1.7501,0.3516) -- (1.7513,0.3711);
\draw[gray] (1.7513,0.3711) -- (1.7536,0.3711);
\draw[gray] (1.7536,0.3711) -- (1.7552,0.3906);
\draw[gray] (1.7552,0.3906) -- (1.7746,0.3906);
\draw[gray] (1.7746,0.3906) -- (1.7771,0.4102);
\draw[gray] (1.7771,0.4102) -- (1.7789,0.4102);
\draw[gray] (1.7789,0.4102) -- (1.7804,0.4297);
\draw[gray] (1.7804,0.4297) -- (1.7850,0.4297);
\draw[gray] (1.7850,0.4297) -- (1.7863,0.4492);
\draw[gray] (1.7863,0.4492) -- (1.7876,0.4492);
\draw[gray] (1.7876,0.4492) -- (1.7895,0.4688);
\draw[gray] (1.7895,0.4688) -- (1.8371,0.4688);
\draw[gray] (1.8371,0.4688) -- (1.8375,0.4883);
\draw[gray] (1.8375,0.4883) -- (1.8382,0.4883);
\draw[gray] (1.8382,0.4883) -- (1.8392,0.5078);
\draw[gray] (1.8392,0.5078) -- (1.8418,0.5078);
\draw[gray] (1.8418,0.5078) -- (1.8436,0.5273);
\draw[gray] (1.8436,0.5273) -- (1.8449,0.5273);
\draw[gray] (1.8449,0.5273) -- (1.8458,0.5469);
\draw[gray] (1.8458,0.5469) -- (1.8523,0.5469);
\draw[gray] (1.8523,0.5469) -- (1.8534,0.5664);
\draw[gray] (1.8534,0.5664) -- (1.8544,0.5664);
\draw[gray] (1.8544,0.5664) -- (1.8559,0.5859);
\draw[gray] (1.8559,0.5859) -- (1.8587,0.5859);
\draw[gray] (1.8587,0.5859) -- (1.8599,0.6055);
\draw[gray] (1.8599,0.6055) -- (1.8624,0.6055);
\draw[gray] (1.8624,0.6055) -- (1.8637,0.6250);
\draw[gray] (1.8637,0.6250) -- (2.4319,0.6250);
\draw[gray] (2.4319,0.6250) -- (2.4336,0.6445);
\draw[gray] (2.4336,0.6445) -- (2.4347,0.6445);
\draw[gray] (2.4347,0.6445) -- (2.4373,0.6641);
\draw[gray] (2.4373,0.6641) -- (2.4396,0.6641);
\draw[gray] (2.4396,0.6641) -- (2.4420,0.6836);
\draw[gray] (2.4420,0.6836) -- (2.4435,0.6836);
\draw[gray] (2.4435,0.6836) -- (2.4449,0.7031);
\draw[gray] (2.4449,0.7031) -- (2.4535,0.7031);
\draw[gray] (2.4535,0.7031) -- (2.4543,0.7227);
\draw[gray] (2.4543,0.7227) -- (2.4546,0.7227);
\draw[gray] (2.4546,0.7227) -- (2.4554,0.7422);
\draw[gray] (2.4554,0.7422) -- (2.4585,0.7422);
\draw[gray] (2.4585,0.7422) -- (2.4596,0.7617);
\draw[gray] (2.4596,0.7617) -- (2.4604,0.7617);
\draw[gray] (2.4604,0.7617) -- (2.4612,0.7813);
\draw[gray] (2.4612,0.7813) -- (2.4957,0.7813);
\draw[gray] (2.4957,0.7813) -- (2.4975,0.8008);
\draw[gray] (2.4975,0.8008) -- (2.4983,0.8008);
\draw[gray] (2.4983,0.8008) -- (2.5000,0.8203);
\draw[gray] (2.5000,0.8203) -- (2.5021,0.8203);
\draw[gray] (2.5021,0.8203) -- (2.5031,0.8398);
\draw[gray] (2.5031,0.8398) -- (2.5053,0.8398);
\draw[gray] (2.5053,0.8398) -- (2.5063,0.8594);
\draw[gray] (2.5063,0.8594) -- (2.5231,0.8594);
\draw[gray] (2.5231,0.8594) -- (2.5244,0.8789);
\draw[gray] (2.5244,0.8789) -- (2.5272,0.8789);
\draw[gray] (2.5272,0.8789) -- (2.5284,0.8984);
\draw[gray] (2.5284,0.8984) -- (2.5325,0.8984);
\draw[gray] (2.5325,0.8984) -- (2.5345,0.9180);
\draw[gray] (2.5345,0.9180) -- (2.5375,0.9180);
\draw[gray] (2.5375,0.9180) -- (2.5393,0.9375);
\draw[gray] (2.5393,0.9375) -- (2.6649,0.9375);
\draw[gray] (2.6649,0.9375) -- (2.6659,0.9570);
\draw[gray] (2.6659,0.9570) -- (2.6663,0.9570);
\draw[gray] (2.6663,0.9570) -- (2.6672,0.9766);
\draw[gray] (2.6672,0.9766) -- (2.6704,0.9766);
\draw[gray] (2.6704,0.9766) -- (2.6709,0.9961);
\draw[gray] (2.6709,0.9961) -- (2.6716,0.9961);
\draw[gray] (2.6716,0.9961) -- (2.6721,1.0156);
\draw[gray] (2.6721,1.0156) -- (2.6879,1.0156);
\draw[gray] (2.6879,1.0156) -- (2.6887,1.0352);
\draw[gray] (2.6887,1.0352) -- (2.6890,1.0352);
\draw[gray] (2.6890,1.0352) -- (2.6899,1.0547);
\draw[gray] (2.6899,1.0547) -- (2.6922,1.0547);
\draw[gray] (2.6922,1.0547) -- (2.6931,1.0742);
\draw[gray] (2.6931,1.0742) -- (2.6944,1.0742);
\draw[gray] (2.6944,1.0742) -- (2.6954,1.0938);
\draw[gray] (2.6954,1.0938) -- (2.7348,1.0938);
\draw[gray] (2.7348,1.0938) -- (2.7359,1.1133);
\draw[gray] (2.7359,1.1133) -- (2.7379,1.1133);
\draw[gray] (2.7379,1.1133) -- (2.7388,1.1328);
\draw[gray] (2.7388,1.1328) -- (2.7425,1.1328);
\draw[gray] (2.7425,1.1328) -- (2.7436,1.1523);
\draw[gray] (2.7436,1.1523) -- (2.7441,1.1523);
\draw[gray] (2.7441,1.1523) -- (2.7449,1.1719);
\draw[gray] (2.7449,1.1719) -- (2.7486,1.1719);
\draw[gray] (2.7486,1.1719) -- (2.7494,1.1914);
\draw[gray] (2.7494,1.1914) -- (2.7499,1.1914);
\draw[gray] (2.7499,1.1914) -- (2.7505,1.2109);
\draw[gray] (2.7505,1.2109) -- (2.7552,1.2109);
\draw[gray] (2.7552,1.2109) -- (2.7556,1.2305);
\draw[gray] (2.7556,1.2305) -- (2.7566,1.2305);
\draw[gray] (2.7566,1.2305) -- (2.7571,1.2500);
\draw[gray] (2.7571,1.2500) -- (3.4846,1.2500);
\draw[gray] (3.4846,1.2500) -- (3.4851,1.2695);
\draw[gray] (3.4851,1.2695) -- (3.4859,1.2695);
\draw[gray] (3.4859,1.2695) -- (3.4866,1.2891);
\draw[gray] (3.4866,1.2891) -- (3.4875,1.2891);
\draw[gray] (3.4875,1.2891) -- (3.4883,1.3086);
\draw[gray] (3.4883,1.3086) -- (3.4888,1.3086);
\draw[gray] (3.4888,1.3086) -- (3.4894,1.3281);
\draw[gray] (3.4894,1.3281) -- (3.4957,1.3281);
\draw[gray] (3.4957,1.3281) -- (3.4963,1.3477);
\draw[gray] (3.4963,1.3477) -- (3.4967,1.3477);
\draw[gray] (3.4967,1.3477) -- (3.4971,1.3672);
\draw[gray] (3.4971,1.3672) -- (3.4985,1.3672);
\draw[gray] (3.4985,1.3672) -- (3.4989,1.3867);
\draw[gray] (3.4989,1.3867) -- (3.4992,1.3867);
\draw[gray] (3.4992,1.3867) -- (3.4998,1.4063);
\draw[gray] (3.4998,1.4063) -- (3.5138,1.4063);
\draw[gray] (3.5138,1.4063) -- (3.5145,1.4258);
\draw[gray] (3.5145,1.4258) -- (3.5152,1.4258);
\draw[gray] (3.5152,1.4258) -- (3.5155,1.4453);
\draw[gray] (3.5155,1.4453) -- (3.5169,1.4453);
\draw[gray] (3.5169,1.4453) -- (3.5174,1.4648);
\draw[gray] (3.5174,1.4648) -- (3.5180,1.4648);
\draw[gray] (3.5180,1.4648) -- (3.5184,1.4844);
\draw[gray] (3.5184,1.4844) -- (3.5246,1.4844);
\draw[gray] (3.5246,1.4844) -- (3.5256,1.5039);
\draw[gray] (3.5256,1.5039) -- (3.5260,1.5039);
\draw[gray] (3.5260,1.5039) -- (3.5269,1.5234);
\draw[gray] (3.5269,1.5234) -- (3.5291,1.5234);
\draw[gray] (3.5291,1.5234) -- (3.5295,1.5430);
\draw[gray] (3.5295,1.5430) -- (3.5302,1.5430);
\draw[gray] (3.5302,1.5430) -- (3.5308,1.5625);
\draw[gray] (3.5308,1.5625) -- (3.6267,1.5625);
\draw[gray] (3.6267,1.5625) -- (3.6271,1.5820);
\draw[gray] (3.6271,1.5820) -- (3.6274,1.5820);
\draw[gray] (3.6274,1.5820) -- (3.6280,1.6016);
\draw[gray] (3.6280,1.6016) -- (3.6304,1.6016);
\draw[gray] (3.6304,1.6016) -- (3.6308,1.6211);
\draw[gray] (3.6308,1.6211) -- (3.6315,1.6211);
\draw[gray] (3.6315,1.6211) -- (3.6321,1.6406);
\draw[gray] (3.6321,1.6406) -- (3.6409,1.6406);
\draw[gray] (3.6409,1.6406) -- (3.6415,1.6602);
\draw[gray] (3.6415,1.6602) -- (3.6419,1.6602);
\draw[gray] (3.6419,1.6602) -- (3.6422,1.6797);
\draw[gray] (3.6422,1.6797) -- (3.6433,1.6797);
\draw[gray] (3.6433,1.6797) -- (3.6437,1.6992);
\draw[gray] (3.6437,1.6992) -- (3.6443,1.6992);
\draw[gray] (3.6443,1.6992) -- (3.6446,1.7188);
\draw[gray] (3.6446,1.7188) -- (3.6532,1.7188);
\draw[gray] (3.6532,1.7188) -- (3.6535,1.7383);
\draw[gray] (3.6535,1.7383) -- (3.6539,1.7383);
\draw[gray] (3.6539,1.7383) -- (3.6544,1.7578);
\draw[gray] (3.6544,1.7578) -- (3.6558,1.7578);
\draw[gray] (3.6558,1.7578) -- (3.6560,1.7773);
\draw[gray] (3.6560,1.7773) -- (3.6565,1.7773);
\draw[gray] (3.6565,1.7773) -- (3.6568,1.7969);
\draw[gray] (3.6568,1.7969) -- (3.6608,1.7969);
\draw[gray] (3.6608,1.7969) -- (3.6612,1.8164);
\draw[gray] (3.6612,1.8164) -- (3.6613,1.8164);
\draw[gray] (3.6613,1.8164) -- (3.6617,1.8359);
\draw[gray] (3.6617,1.8359) -- (3.6636,1.8359);
\draw[gray] (3.6636,1.8359) -- (3.6643,1.8555);
\draw[gray] (3.6643,1.8555) -- (3.6644,1.8555);
\draw[gray] (3.6644,1.8555) -- (3.6650,1.8750);
\draw[gray] (3.6650,1.8750) -- (3.9663,1.8750);
\draw[gray] (3.9663,1.8750) -- (3.9668,1.8945);
\draw[gray] (3.9668,1.8945) -- (3.9678,1.8945);
\draw[gray] (3.9678,1.8945) -- (3.9682,1.9141);
\draw[gray] (3.9682,1.9141) -- (3.9706,1.9141);
\draw[gray] (3.9706,1.9141) -- (3.9711,1.9336);
\draw[gray] (3.9711,1.9336) -- (3.9714,1.9336);
\draw[gray] (3.9714,1.9336) -- (3.9721,1.9531);
\draw[gray] (3.9721,1.9531) -- (3.9846,1.9531);
\draw[gray] (3.9846,1.9531) -- (3.9854,1.9727);
\draw[gray] (3.9854,1.9727) -- (3.9855,1.9727);
\draw[gray] (3.9855,1.9727) -- (3.9863,1.9922);
\draw[gray] (3.9863,1.9922) -- (3.9885,1.9922);
\draw[gray] (3.9885,1.9922) -- (3.9890,2.0117);
\draw[gray] (3.9890,2.0117) -- (3.9900,2.0117);
\draw[gray] (3.9900,2.0117) -- (3.9906,2.0313);
\draw[gray] (3.9906,2.0313) -- (4.0057,2.0313);
\draw[gray] (4.0057,2.0313) -- (4.0062,2.0508);
\draw[gray] (4.0062,2.0508) -- (4.0064,2.0508);
\draw[gray] (4.0064,2.0508) -- (4.0070,2.0703);
\draw[gray] (4.0070,2.0703) -- (4.0091,2.0703);
\draw[gray] (4.0091,2.0703) -- (4.0099,2.0898);
\draw[gray] (4.0099,2.0898) -- (4.0102,2.0898);
\draw[gray] (4.0102,2.0898) -- (4.0111,2.1094);
\draw[gray] (4.0111,2.1094) -- (4.0146,2.1094);
\draw[gray] (4.0146,2.1094) -- (4.0149,2.1289);
\draw[gray] (4.0149,2.1289) -- (4.0152,2.1289);
\draw[gray] (4.0152,2.1289) -- (4.0154,2.1484);
\draw[gray] (4.0154,2.1484) -- (4.0160,2.1484);
\draw[gray] (4.0160,2.1484) -- (4.0164,2.1680);
\draw[gray] (4.0164,2.1680) -- (4.0167,2.1680);
\draw[gray] (4.0167,2.1680) -- (4.0171,2.1875);
\draw[gray] (4.0171,2.1875) -- (4.1113,2.1875);
\draw[gray] (4.1113,2.1875) -- (4.1117,2.2070);
\draw[gray] (4.1117,2.2070) -- (4.1119,2.2070);
\draw[gray] (4.1119,2.2070) -- (4.1123,2.2266);
\draw[gray] (4.1123,2.2266) -- (4.1132,2.2266);
\draw[gray] (4.1132,2.2266) -- (4.1136,2.2461);
\draw[gray] (4.1136,2.2461) -- (4.1139,2.2461);
\draw[gray] (4.1139,2.2461) -- (4.1145,2.2656);
\draw[gray] (4.1145,2.2656) -- (4.1172,2.2656);
\draw[gray] (4.1172,2.2656) -- (4.1178,2.2852);
\draw[gray] (4.1178,2.2852) -- (4.1182,2.2852);
\draw[gray] (4.1182,2.2852) -- (4.1187,2.3047);
\draw[gray] (4.1187,2.3047) -- (4.1207,2.3047);
\draw[gray] (4.1207,2.3047) -- (4.1213,2.3242);
\draw[gray] (4.1213,2.3242) -- (4.1217,2.3242);
\draw[gray] (4.1217,2.3242) -- (4.1221,2.3438);
\draw[gray] (4.1221,2.3438) -- (4.1487,2.3438);
\draw[gray] (4.1487,2.3438) -- (4.1490,2.3633);
\draw[gray] (4.1490,2.3633) -- (4.1497,2.3633);
\draw[gray] (4.1497,2.3633) -- (4.1500,2.3828);
\draw[gray] (4.1500,2.3828) -- (4.1528,2.3828);
\draw[gray] (4.1528,2.3828) -- (4.1533,2.4023);
\draw[gray] (4.1533,2.4023) -- (4.1537,2.4023);
\draw[gray] (4.1537,2.4023) -- (4.1540,2.4219);
\draw[gray] (4.1540,2.4219) -- (4.1555,2.4219);
\draw[gray] (4.1555,2.4219) -- (4.1557,2.4414);
\draw[gray] (4.1557,2.4414) -- (4.1561,2.4414);
\draw[gray] (4.1561,2.4414) -- (4.1566,2.4609);
\draw[gray] (4.1566,2.4609) -- (4.1582,2.4609);
\draw[gray] (4.1582,2.4609) -- (4.1586,2.4805);
\draw[gray] (4.1586,2.4805) -- (4.1596,2.4805);
\draw[gray] (4.1596,2.4805) -- (4.1600,2.5000);
\draw[gray] (4.1600,2.5000) -- (5.1980,2.5000);
\draw[gray] (5.1980,2.5000) -- (5.2000,2.5195);
\draw[gray] (5.2000,2.5195) -- (5.2013,2.5195);
\draw[gray] (5.2013,2.5195) -- (5.2023,2.5391);
\draw[gray] (5.2023,2.5391) -- (5.2042,2.5391);
\draw[gray] (5.2042,2.5391) -- (5.2050,2.5586);
\draw[gray] (5.2050,2.5586) -- (5.2059,2.5586);
\draw[gray] (5.2059,2.5586) -- (5.2075,2.5781);
\draw[gray] (5.2075,2.5781) -- (5.2199,2.5781);
\draw[gray] (5.2199,2.5781) -- (5.2206,2.5977);
\draw[gray] (5.2206,2.5977) -- (5.2209,2.5977);
\draw[gray] (5.2209,2.5977) -- (5.2218,2.6172);
\draw[gray] (5.2218,2.6172) -- (5.2246,2.6172);
\draw[gray] (5.2246,2.6172) -- (5.2254,2.6367);
\draw[gray] (5.2254,2.6367) -- (5.2260,2.6367);
\draw[gray] (5.2260,2.6367) -- (5.2273,2.6563);
\draw[gray] (5.2273,2.6563) -- (5.2466,2.6563);
\draw[gray] (5.2466,2.6563) -- (5.2475,2.6758);
\draw[gray] (5.2475,2.6758) -- (5.2489,2.6758);
\draw[gray] (5.2489,2.6758) -- (5.2507,2.6953);
\draw[gray] (5.2507,2.6953) -- (5.2575,2.6953);
\draw[gray] (5.2575,2.6953) -- (5.2595,2.7148);
\draw[gray] (5.2595,2.7148) -- (5.2600,2.7148);
\draw[gray] (5.2600,2.7148) -- (5.2621,2.7344);
\draw[gray] (5.2621,2.7344) -- (5.2758,2.7344);
\draw[gray] (5.2758,2.7344) -- (5.2771,2.7539);
\draw[gray] (5.2771,2.7539) -- (5.2784,2.7539);
\draw[gray] (5.2784,2.7539) -- (5.2798,2.7734);
\draw[gray] (5.2798,2.7734) -- (5.2818,2.7734);
\draw[gray] (5.2818,2.7734) -- (5.2824,2.7930);
\draw[gray] (5.2824,2.7930) -- (5.2836,2.7930);
\draw[gray] (5.2836,2.7930) -- (5.2842,2.8125);
\draw[gray] (5.2842,2.8125) -- (5.3741,2.8125);
\draw[gray] (5.3741,2.8125) -- (5.3746,2.8320);
\draw[gray] (5.3746,2.8320) -- (5.3754,2.8320);
\draw[gray] (5.3754,2.8320) -- (5.3761,2.8516);
\draw[gray] (5.3761,2.8516) -- (5.3775,2.8516);
\draw[gray] (5.3775,2.8516) -- (5.3782,2.8711);
\draw[gray] (5.3782,2.8711) -- (5.3791,2.8711);
\draw[gray] (5.3791,2.8711) -- (5.3804,2.8906);
\draw[gray] (5.3804,2.8906) -- (5.3857,2.8906);
\draw[gray] (5.3857,2.8906) -- (5.3874,2.9102);
\draw[gray] (5.3874,2.9102) -- (5.3882,2.9102);
\draw[gray] (5.3882,2.9102) -- (5.3892,2.9297);
\draw[gray] (5.3892,2.9297) -- (5.3916,2.9297);
\draw[gray] (5.3916,2.9297) -- (5.3929,2.9492);
\draw[gray] (5.3929,2.9492) -- (5.3934,2.9492);
\draw[gray] (5.3934,2.9492) -- (5.3951,2.9688);
\draw[gray] (5.3951,2.9688) -- (5.4140,2.9688);
\draw[gray] (5.4140,2.9688) -- (5.4144,2.9883);
\draw[gray] (5.4144,2.9883) -- (5.4146,2.9883);
\draw[gray] (5.4146,2.9883) -- (5.4148,3.0078);
\draw[gray] (5.4148,3.0078) -- (5.4159,3.0078);
\draw[gray] (5.4159,3.0078) -- (5.4164,3.0273);
\draw[gray] (5.4164,3.0273) -- (5.4169,3.0273);
\draw[gray] (5.4169,3.0273) -- (5.4176,3.0469);
\draw[gray] (5.4176,3.0469) -- (5.4225,3.0469);
\draw[gray] (5.4225,3.0469) -- (5.4234,3.0664);
\draw[gray] (5.4234,3.0664) -- (5.4242,3.0664);
\draw[gray] (5.4242,3.0664) -- (5.4247,3.0859);
\draw[gray] (5.4247,3.0859) -- (5.4256,3.0859);
\draw[gray] (5.4256,3.0859) -- (5.4262,3.1055);
\draw[gray] (5.4262,3.1055) -- (5.4271,3.1055);
\draw[gray] (5.4271,3.1055) -- (5.4280,3.1250);
\draw[gray] (5.4280,3.1250) -- (5.6328,3.1250);
\draw[gray] (5.6328,3.1250) -- (5.6335,3.1445);
\draw[gray] (5.6335,3.1445) -- (5.6343,3.1445);
\draw[gray] (5.6343,3.1445) -- (5.6355,3.1641);
\draw[gray] (5.6355,3.1641) -- (5.6371,3.1641);
\draw[gray] (5.6371,3.1641) -- (5.6389,3.1836);
\draw[gray] (5.6389,3.1836) -- (5.6397,3.1836);
\draw[gray] (5.6397,3.1836) -- (5.6410,3.2031);
\draw[gray] (5.6410,3.2031) -- (5.6478,3.2031);
\draw[gray] (5.6478,3.2031) -- (5.6490,3.2227);
\draw[gray] (5.6490,3.2227) -- (5.6505,3.2227);
\draw[gray] (5.6505,3.2227) -- (5.6518,3.2422);
\draw[gray] (5.6518,3.2422) -- (5.6531,3.2422);
\draw[gray] (5.6531,3.2422) -- (5.6547,3.2617);
\draw[gray] (5.6547,3.2617) -- (5.6550,3.2617);
\draw[gray] (5.6550,3.2617) -- (5.6565,3.2813);
\draw[gray] (5.6565,3.2813) -- (5.6774,3.2813);
\draw[gray] (5.6774,3.2813) -- (5.6782,3.3008);
\draw[gray] (5.6782,3.3008) -- (5.6796,3.3008);
\draw[gray] (5.6796,3.3008) -- (5.6803,3.3203);
\draw[gray] (5.6803,3.3203) -- (5.6810,3.3203);
\draw[gray] (5.6810,3.3203) -- (5.6823,3.3398);
\draw[gray] (5.6823,3.3398) -- (5.6827,3.3398);
\draw[gray] (5.6827,3.3398) -- (5.6843,3.3594);
\draw[gray] (5.6843,3.3594) -- (5.6887,3.3594);
\draw[gray] (5.6887,3.3594) -- (5.6889,3.3789);
\draw[gray] (5.6889,3.3789) -- (5.6892,3.3789);
\draw[gray] (5.6892,3.3789) -- (5.6897,3.3984);
\draw[gray] (5.6897,3.3984) -- (5.6909,3.3984);
\draw[gray] (5.6909,3.3984) -- (5.6913,3.4180);
\draw[gray] (5.6913,3.4180) -- (5.6913,3.4180);
\draw[gray] (5.6913,3.4180) -- (5.6917,3.4375);
\draw[gray] (5.6917,3.4375) -- (5.7053,3.4375);
\draw[gray] (5.7053,3.4375) -- (5.7058,3.4570);
\draw[gray] (5.7058,3.4570) -- (5.7061,3.4570);
\draw[gray] (5.7061,3.4570) -- (5.7066,3.4766);
\draw[gray] (5.7066,3.4766) -- (5.7103,3.4766);
\draw[gray] (5.7103,3.4766) -- (5.7108,3.4961);
\draw[gray] (5.7108,3.4961) -- (5.7112,3.4961);
\draw[gray] (5.7112,3.4961) -- (5.7116,3.5156);
\draw[gray] (5.7116,3.5156) -- (5.7175,3.5156);
\draw[gray] (5.7175,3.5156) -- (5.7178,3.5352);
\draw[gray] (5.7178,3.5352) -- (5.7180,3.5352);
\draw[gray] (5.7180,3.5352) -- (5.7184,3.5547);
\draw[gray] (5.7184,3.5547) -- (5.7195,3.5547);
\draw[gray] (5.7195,3.5547) -- (5.7201,3.5742);
\draw[gray] (5.7201,3.5742) -- (5.7204,3.5742);
\draw[gray] (5.7204,3.5742) -- (5.7210,3.5938);
\draw[gray] (5.7210,3.5938) -- (5.7468,3.5938);
\draw[gray] (5.7468,3.5938) -- (5.7478,3.6133);
\draw[gray] (5.7478,3.6133) -- (5.7485,3.6133);
\draw[gray] (5.7485,3.6133) -- (5.7491,3.6328);
\draw[gray] (5.7491,3.6328) -- (5.7530,3.6328);
\draw[gray] (5.7530,3.6328) -- (5.7548,3.6523);
\draw[gray] (5.7548,3.6523) -- (5.7561,3.6523);
\draw[gray] (5.7561,3.6523) -- (5.7579,3.6719);
\draw[gray] (5.7579,3.6719) -- (5.7666,3.6719);
\draw[gray] (5.7666,3.6719) -- (5.7670,3.6914);
\draw[gray] (5.7670,3.6914) -- (5.7674,3.6914);
\draw[gray] (5.7674,3.6914) -- (5.7679,3.7109);
\draw[gray] (5.7679,3.7109) -- (5.7706,3.7109);
\draw[gray] (5.7706,3.7109) -- (5.7709,3.7305);
\draw[gray] (5.7709,3.7305) -- (5.7713,3.7305);
\draw[gray] (5.7713,3.7305) -- (5.7719,3.7500);
\draw[line width=1pt, black] (5.7719,3.7500) -- (6.8772,3.7500);
\draw[gray] (6.8772,3.7500) -- (6.8781,3.7695);
\draw[gray] (6.8781,3.7695) -- (6.8793,3.7695);
\draw[gray] (6.8793,3.7695) -- (6.8812,3.7891);
\draw[gray] (6.8812,3.7891) -- (6.8819,3.7891);
\draw[gray] (6.8819,3.7891) -- (6.8827,3.8086);
\draw[gray] (6.8827,3.8086) -- (6.8840,3.8086);
\draw[gray] (6.8840,3.8086) -- (6.8856,3.8281);
\draw[gray] (6.8856,3.8281) -- (6.8931,3.8281);
\draw[gray] (6.8931,3.8281) -- (6.8946,3.8477);
\draw[gray] (6.8946,3.8477) -- (6.8955,3.8477);
\draw[gray] (6.8955,3.8477) -- (6.8969,3.8672);
\draw[gray] (6.8969,3.8672) -- (6.9002,3.8672);
\draw[gray] (6.9002,3.8672) -- (6.9019,3.8867);
\draw[gray] (6.9019,3.8867) -- (6.9024,3.8867);
\draw[gray] (6.9024,3.8867) -- (6.9040,3.9063);
\draw[gray] (6.9040,3.9063) -- (6.9635,3.9063);
\draw[gray] (6.9635,3.9063) -- (6.9642,3.9258);
\draw[gray] (6.9642,3.9258) -- (6.9645,3.9258);
\draw[gray] (6.9645,3.9258) -- (6.9653,3.9453);
\draw[gray] (6.9653,3.9453) -- (6.9672,3.9453);
\draw[gray] (6.9672,3.9453) -- (6.9678,3.9648);
\draw[gray] (6.9678,3.9648) -- (6.9685,3.9648);
\draw[gray] (6.9685,3.9648) -- (6.9695,3.9844);
\draw[gray] (6.9695,3.9844) -- (6.9799,3.9844);
\draw[gray] (6.9799,3.9844) -- (6.9809,4.0039);
\draw[gray] (6.9809,4.0039) -- (6.9817,4.0039);
\draw[gray] (6.9817,4.0039) -- (6.9826,4.0234);
\draw[gray] (6.9826,4.0234) -- (6.9841,4.0234);
\draw[gray] (6.9841,4.0234) -- (6.9854,4.0430);
\draw[gray] (6.9854,4.0430) -- (6.9860,4.0430);
\draw[gray] (6.9860,4.0430) -- (6.9875,4.0625);
\draw[gray] (6.9875,4.0625) -- (7.0923,4.0625);
\draw[gray] (7.0923,4.0625) -- (7.0947,4.0820);
\draw[gray] (7.0947,4.0820) -- (7.0960,4.0820);
\draw[gray] (7.0960,4.0820) -- (7.0974,4.1016);
\draw[gray] (7.0974,4.1016) -- (7.1036,4.1016);
\draw[gray] (7.1036,4.1016) -- (7.1050,4.1211);
\draw[gray] (7.1050,4.1211) -- (7.1060,4.1211);
\draw[gray] (7.1060,4.1211) -- (7.1072,4.1406);
\draw[gray] (7.1072,4.1406) -- (7.1216,4.1406);
\draw[gray] (7.1216,4.1406) -- (7.1233,4.1602);
\draw[gray] (7.1233,4.1602) -- (7.1254,4.1602);
\draw[gray] (7.1254,4.1602) -- (7.1270,4.1797);
\draw[gray] (7.1270,4.1797) -- (7.1314,4.1797);
\draw[gray] (7.1314,4.1797) -- (7.1326,4.1992);
\draw[gray] (7.1326,4.1992) -- (7.1336,4.1992);
\draw[gray] (7.1336,4.1992) -- (7.1344,4.2188);
\draw[gray] (7.1344,4.2188) -- (7.1925,4.2188);
\draw[gray] (7.1925,4.2188) -- (7.1936,4.2383);
\draw[gray] (7.1936,4.2383) -- (7.1949,4.2383);
\draw[gray] (7.1949,4.2383) -- (7.1965,4.2578);
\draw[gray] (7.1965,4.2578) -- (7.2010,4.2578);
\draw[gray] (7.2010,4.2578) -- (7.2033,4.2773);
\draw[gray] (7.2033,4.2773) -- (7.2052,4.2773);
\draw[gray] (7.2052,4.2773) -- (7.2084,4.2969);
\draw[gray] (7.2084,4.2969) -- (7.2215,4.2969);
\draw[gray] (7.2215,4.2969) -- (7.2222,4.3164);
\draw[gray] (7.2222,4.3164) -- (7.2229,4.3164);
\draw[gray] (7.2229,4.3164) -- (7.2240,4.3359);
\draw[gray] (7.2240,4.3359) -- (7.2292,4.3359);
\draw[gray] (7.2292,4.3359) -- (7.2301,4.3555);
\draw[gray] (7.2301,4.3555) -- (7.2306,4.3555);
\draw[gray] (7.2306,4.3555) -- (7.2315,4.3750);
\draw[gray] (7.2315,4.3750) -- (7.5763,4.3750);
\draw[gray] (7.5763,4.3750) -- (7.5774,4.3945);
\draw[gray] (7.5774,4.3945) -- (7.5791,4.3945);
\draw[gray] (7.5791,4.3945) -- (7.5804,4.4141);
\draw[gray] (7.5804,4.4141) -- (7.5853,4.4141);
\draw[gray] (7.5853,4.4141) -- (7.5866,4.4336);
\draw[gray] (7.5866,4.4336) -- (7.5871,4.4336);
\draw[gray] (7.5871,4.4336) -- (7.5883,4.4531);
\draw[gray] (7.5883,4.4531) -- (7.6058,4.4531);
\draw[gray] (7.6058,4.4531) -- (7.6086,4.4727);
\draw[gray] (7.6086,4.4727) -- (7.6117,4.4727);
\draw[gray] (7.6117,4.4727) -- (7.6136,4.4922);
\draw[gray] (7.6136,4.4922) -- (7.6183,4.4922);
\draw[gray] (7.6183,4.4922) -- (7.6201,4.5117);
\draw[gray] (7.6201,4.5117) -- (7.6207,4.5117);
\draw[gray] (7.6207,4.5117) -- (7.6225,4.5313);
\draw[gray] (7.6225,4.5313) -- (7.6786,4.5313);
\draw[gray] (7.6786,4.5313) -- (7.6800,4.5508);
\draw[gray] (7.6800,4.5508) -- (7.6817,4.5508);
\draw[gray] (7.6817,4.5508) -- (7.6827,4.5703);
\draw[gray] (7.6827,4.5703) -- (7.6855,4.5703);
\draw[gray] (7.6855,4.5703) -- (7.6883,4.5898);
\draw[gray] (7.6883,4.5898) -- (7.6893,4.5898);
\draw[gray] (7.6893,4.5898) -- (7.6913,4.6094);
\draw[gray] (7.6913,4.6094) -- (7.6973,4.6094);
\draw[gray] (7.6973,4.6094) -- (7.6983,4.6289);
\draw[gray] (7.6983,4.6289) -- (7.6998,4.6289);
\draw[gray] (7.6998,4.6289) -- (7.7014,4.6484);
\draw[gray] (7.7014,4.6484) -- (7.7043,4.6484);
\draw[gray] (7.7043,4.6484) -- (7.7052,4.6680);
\draw[gray] (7.7052,4.6680) -- (7.7058,4.6680);
\draw[gray] (7.7058,4.6680) -- (7.7064,4.6875);
\draw[gray] (7.7064,4.6875) -- (7.8494,4.6875);
\draw[gray] (7.8494,4.6875) -- (7.8506,4.7070);
\draw[gray] (7.8506,4.7070) -- (7.8519,4.7070);
\draw[gray] (7.8519,4.7070) -- (7.8540,4.7266);
\draw[gray] (7.8540,4.7266) -- (7.8609,4.7266);
\draw[gray] (7.8609,4.7266) -- (7.8635,4.7461);
\draw[gray] (7.8635,4.7461) -- (7.8657,4.7461);
\draw[gray] (7.8657,4.7461) -- (7.8696,4.7656);
\draw[gray] (7.8696,4.7656) -- (7.8855,4.7656);
\draw[gray] (7.8855,4.7656) -- (7.8865,4.7852);
\draw[gray] (7.8865,4.7852) -- (7.8880,4.7852);
\draw[gray] (7.8880,4.7852) -- (7.8894,4.8047);
\draw[gray] (7.8894,4.8047) -- (7.8983,4.8047);
\draw[gray] (7.8983,4.8047) -- (7.8992,4.8242);
\draw[gray] (7.8992,4.8242) -- (7.9015,4.8242);
\draw[gray] (7.9015,4.8242) -- (7.9025,4.8438);
\draw[gray] (7.9025,4.8438) -- (7.9389,4.8438);
\draw[gray] (7.9389,4.8438) -- (7.9415,4.8633);
\draw[gray] (7.9415,4.8633) -- (7.9424,4.8633);
\draw[gray] (7.9424,4.8633) -- (7.9454,4.8828);
\draw[gray] (7.9454,4.8828) -- (7.9470,4.8828);
\draw[gray] (7.9470,4.8828) -- (7.9502,4.9023);
\draw[gray] (7.9502,4.9023) -- (7.9510,4.9023);
\draw[gray] (7.9510,4.9023) -- (7.9537,4.9219);
\draw[gray] (7.9537,4.9219) -- (7.9805,4.9219);
\draw[gray] (7.9805,4.9219) -- (7.9831,4.9414);
\draw[gray] (7.9831,4.9414) -- (7.9852,4.9414);
\draw[gray] (7.9852,4.9414) -- (7.9870,4.9609);
\draw[gray] (7.9870,4.9609) -- (7.9934,4.9609);
\draw[gray] (7.9934,4.9609) -- (7.9953,4.9805);
\draw[gray] (7.9953,4.9805) -- (7.9981,4.9805);
\draw[gray] (7.9981,4.9805) -- (8.0000,5.0000);
\draw (0.0000,0.0000) -- (8.0000,5.0000);
\node at (2.5000,2.2000) {$\rho \cdot \nu (I_{j1})$};
\node at (6.5000,2.1875) {$\nu (\rho \cdot_L (G_j \cup I_{j1}))$};
\draw[->] (-0.5000,0.0000) -- (8.7000,0.0000);
\draw[->] (0.0000,-0.5000) -- (0.0000,5.7000);
\node at (9.0000,0.0000) {$\rho$};
\node at (8.0000,-0.3000) {\small{$1$}};
\draw[-] (8.0000,-0.1000) -- (8.0000,0.1000);
\node at (-0.6000,5.0000) {\footnotesize{$\nu (I_{j1})$}};
\draw[-] (-0.1000,5.0000) -- (0.1000,5.0000);
\node at (-0.6000,2.5000) {$\frac{\nu (I_{j1})}{2}$};
\draw[-] (-0.1000,2.5000) -- (0.1000,2.5000);
\end{tikzpicture}

\caption{The black diagonal line shows the \textsc{rhs} of equation~\ref{eq:firstequiv} and the grey line the \textsc{lhs} of the same equation. Note in particular that the \textsc{lhs} is constant for $\rho$ corresponding to gaps in the Cantor set.}
\label{fig:stair}
\end{figure}
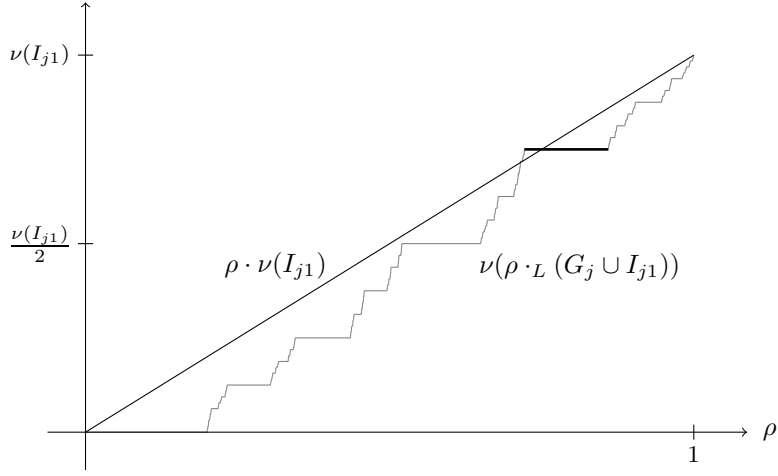
As the \textsc{lhs} of equation~\ref{eq:noti} is constant at $\rho$ corresponding to the gaps associated to \( C \) (see Figure~\ref{fig:stair}), we now conclude that there must exist at least one pair \( (k, m ) \), where \( k \) is a binary word and \( m \) is a positive integer, such that  \( {\rho =  \rho_{j1k10^m}  }\) satisfies  the inequality in equation~\ref{eq:noti} and, in addition, such that if \( (k', m' ) \) is any other pair for which \( \rho = \rho_{j1k'10^{m'}} \) satisfies the inequality in equation~\ref{eq:noti}, then \( |k| \leq |k'| \).

  As \( \rho = \rho_{j1k10^m} \) minimizes $|k|$, the inequality in equation~\ref{eq:firstequiv} holds for \({ \rho = \rho_{j1k0}} \) and \( {\rho = \rho_{j1k1}}\), i.e.\
\[ \nu (\rho_{j1k0}  \cdot_L (G_j \cup I_{j1})) \leq \rho_{j1k0}  \cdot \nu (I_{j1})\]
  and
\[ \nu (\rho_{j1k1} \cdot_L (G_j \cup I_{j1})) \leq \rho_{j1k1} \cdot \nu (I_{j1})\]
   This implies that the line segment between the two points 
\[  (\rho_{j1k0}, \nu (\rho_{j1k0} \cdot_L (G_j \cup I_{j1}))) \quad \text{and} \quad (\rho_{j1k1}, \nu (\rho_{j1k1} \cdot_L (G_j \cup I_{j1}))) \]
 lies completely below the line \( \rho \cdot \nu(I_{j1})\)  for \( \rho \in [\rho_{j1k0},\rho_{j1k1}] \)., i.e.\ we have
\begin{equation*}
\begin{split}
\nu& (\rho_{j1k0} \cdot_L (G_j \cup I_{j1})) + \frac{\rho - \rho_{j1k0}}{\rho_{j1k1} - \rho_{j1k0}} \times \\
&\bigl(  \nu (\rho_{j1k1} \cdot_L (G_j \cup I_{j1})) -  \nu (\rho_{j1k0} \cdot_L (G_j \cup I_{j1})) \bigr)  <  \rho \cdot \nu(I_{j1})
\end{split}
\end{equation*}
  for all  \( \rho \in [\rho_{j1k0},\rho_{j1k1}] \). Noting that 
\[ 
\nu (\rho_{j1k1} \cdot_L (G_j \cup I_{j1})) -  \nu (\rho_{j1k0} \cdot_L (G_j \cup I_{j1}))  = \nu(I_{j1k1}) 
\]
and using equation~\ref{eq:noti} yields
\begin{equation*}
  \nu (\rho_{j1k0} \cdot_L (G_j \cup I_{j1})) + \frac{\rho - \rho_{j1k0}}{\rho_{j1k1} - \rho_{j1k0}} \cdot   \nu(I_{j1k1}) < \nu \left(\rho \cdot_L (G_j \cup I_{j1}) \right).
\end{equation*}
  Now set \( \rho = \rho_{j1k10^m}\). Then \( \rho \in [\rho_{j1k0},\rho_{j1k1}] \) and
\begin{equation*}
\begin{split}
&\hspace{-1em} \nu \left(\rho \cdot_L (G_j \cup I_{j1}) \right)-   \nu (\rho_{j0k0} \cdot_L (G_j \cup I_{j1})) 
\\
&=  \nu \left(\rho_{j1k10^m} \cdot_L (G_j \cup I_{j1}) \right)-   \nu (\rho_{j0k0} \cdot_L (G_j \cup I_{j1})) 
\\
&= \nu(I_{jk10^{m}})
= \frac{1}{2^{m}}  \nu  (I_{j1k1}).
\end{split}
\end{equation*}
  Also, 
\begin{equation*}
\frac{\rho - \rho_{j1k0}}{\rho_{j1k1 - \rho_{j1k0}}} 
= \frac{\rho_{j1k10^m} - \rho_{j1k0}}{\rho_{j1k1 - \rho_{j1k0}}}
= \frac{|G_{j1k} \cup I_{j1k10^{m}}|}{|G_{j1k} \cup I_{j1k1}|}.
\end{equation*}
  Combining the last three equations and dividing by \( \nu (I_{j1k1})\) we obtain
\begin{equation*}
\frac{1}{2^{m}} > \frac{|G_{j1k} \cup I_{j1k10^{m}}|}{|G_{j1k} \cup I_{j1k1}|}.
\end{equation*}
This means that (ii) must be false if (i) is false, which finishes the proof of the lemma.

\end{proof}

In addition to the lemma above, in the proof of theorem \ref{thm:general-1-thm} we will need a lemma which is sometimes called the \textit{mass~distribution principle}. In this paper, we will only use the mass distribution principle for Cantor measures.

\begin{lemma}[The mass distribution principle]
  Let \( \nu \) be a Cantor measure, \( E \subseteq \mathbb{R}\), \( h(\xi, \delta) \) a gauge function and \( q, \varepsilon>0 \) positive numbers such that \(
 h\left(I\right) \geq q \cdot \nu \left(I \right) 
\) for all intervals \( I \) with diameter less that \( \varepsilon \) contained in \( (1+\varepsilon)E \). Then \(
\mu_h \left(E \cap C \right) \geq q \cdot \nu\left( E\right)  
\).
\end{lemma}

\begin{proof}[Proof of the mass distribution principle]
  Fix \( \delta < \varepsilon \) and let \( \{ I_k\}_{k \in K} \) be an arbitrarily chosen \hbox{\( \delta \)-covering} of \( E \). Then
\begin{equation*}
\sum_{k \in K} h\left( I_k\right) \geq \sum_{k \in K} q \cdot \nu \left(I_k \right) \geq q \cdot  \nu \left(E \right)
\end{equation*}
  since \( E \subset \bigcup_{k \in K} I_k \). By letting \( \delta \to 0 \), we get \( \mu_h \left( E \cap C \right) \geq q \cdot  \nu\left(E\right) \).
\end{proof}

We now proceed to the proof of our main theorem.
\begin{proof}[Proof of theorem \ref{thm:general-1-thm}]
  For the upper bound on \( \mu_h(J \cap C) \), consider the covering of \( J \cap C \) with the basic intervals \( I_j \) from some fixed step \( k \) of the construction which intersects \( J \), i.e.\  all basic intervals \( I_j \) associated to \( C \) for which \( I_j \cap J \not = \emptyset \) and \( |j|=k \). Then
  \begin{equation*}
    \begin{split}
      \mu_h (J \cap C) 
      \leq&
      \lim_{k \to \infty} \sum_{\substack{|j|=k\\[0.5mm] I_j \cap J \not = \emptyset}} h(I_j) 
       \leq
 \lim_{k \to \infty} \sum_{\substack{|j|=k
      \\[0.5mm]
      I_j \cap J \not = \emptyset}} r \cdot  \nu(I_j) 
      \\
      =& \lim_{k \to \infty} r \cdot \nu \left( \bigcup_{\substack{|j|=k\\[0.5mm]I_j \cap J \not = \emptyset}} I_j \right) 
    \leq  
\lim_{k \to \infty}r \cdot \nu \left( \bigcup_{\substack{|j|=k
      \\[0.5mm]
      I_j \cap \partial J \not = \emptyset}} I_j \right) 
+ r \cdot \nu \left( J \right) 
    \end{split}
  \end{equation*}
  As at most two basic intervals from any fixed step \( k \) of the construction can intersect \( \partial J \) and \( \nu(I_j) = 2^{-|j|} \) for any basic interval, we get
  \begin{equation*}
    \mu_h (J \cap C) \leq 
    \lim_{k \to \infty} r \cdot \nu \left( J \right) + r \cdot \nu \, \Bigl(\! \bigcup_{\substack{|j|=k\\[0.5mm] I_j \cap \partial J \not = \emptyset}} I_j \Bigr)  \leq 
    \lim_{k \to \infty} r \cdot \nu \left( J \right) + r \cdot 2 \cdot 2^{-k} = 
    r \cdot \nu \left( J \right) 
  \end{equation*}

  We will now show that the lower bound in equation~\ref{eq:important} holds, i.e.\ we will show that 
  \[ {\mu_h(J\cap C) \geq \bigl( q - (r-q) \bigr) \cdot \nu(  J ) }\]
  To do this we will use the mass distribution principle after showing that 
  \( h(I) \geq  \bigl( q - (r-q) \bigr) \cdot  \nu(I) \)
  for all intervals \( I \subseteq J(1 + \varepsilon) \) with \( |I|< \Delta \) for some small \( \Delta>0 \). As \( h \) is a gauge function,  \( h \) is increasing as an interval function and it is therefore enough to consider the case when \( I \) is a near basic interval.

To this end, pick \( \Delta_0 \) small enough for the assumptions of the theorem to hold for all intervals with diameter less than \( \Delta_0 \). As $ |I_j| \to 0$ when $|j| \to 0$, there exists $ k \in \mathbb{N}$ such that $ \max_{|j|>k} |I_j| \leq \Delta_0$. Fix any such $ k$ and set  $\Delta = \min_{|j| \leq k} |G_j|$. Now let \( I = I(w,\delta)\) be any near basic interval associated with \( C \) with \( |I|<\Delta \). If \( j_1 \) and \( j_1 \) are two binary words, we say a the gap \( G_{j_1} \) is older than a gap \( G_{j_2} \) if \( |j_1|>|j_2| \). Let \( G_j \) be the oldest gap which is a subset of \( I \). Since \( G_j \) is the oldest gap contained in \( I \) and \( I \) is a near basic interval, \( I \subseteq I_j \). The choice of \( \Delta \) ensures that the diameter of \( I_j \) is smaller than \( \Delta_0 \), which enables us to use all the assumptions of the theorem in the reasoning below.

  To simplify notations, set \( J_0 = I\cap I_{j0} \) and \( J_1 = I \cap I_{j1} \) and note that \( I \cap C \subseteq J_0 \cup J_1 \). 
  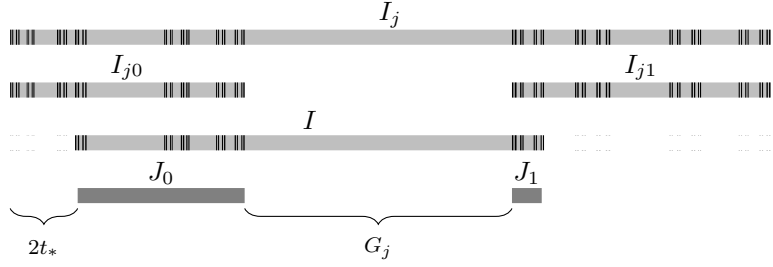
\begin{figure}[h]
    \centering
\begin{tikzpicture}

\node at (1.5422,0.3000) {$I_{j0}$};
\draw[line width=2mm,color=lightgray] (0.0000,0.0000) -- (3.0844,0.0000);

\node at (8.3001,0.3000) {$I_{j1}$};
\draw[line width=2mm,color=lightgray] (6.6002,0.0000) -- (10.0000,0.0000);

\draw[line width=2mm] (0.0000,0.0000) -- (0.0156,0.0000);
\draw[line width=2mm] (0.0275,0.0000) -- (0.0411,0.0000);
\draw[line width=2mm] (0.0762,0.0000) -- (0.0919,0.0000);
\draw[line width=2mm] (0.1077,0.0000) -- (0.1233,0.0000);
\draw[line width=2mm] (0.2181,0.0000) -- (0.2291,0.0000);
\draw[line width=2mm] (0.2393,0.0000) -- (0.2496,0.0000);
\draw[line width=2mm] (0.2838,0.0000) -- (0.2963,0.0000);
\draw[line width=2mm] (0.3057,0.0000) -- (0.3199,0.0000);
\draw[line width=2mm] (0.6164,0.0000) -- (0.6339,0.0000);
\draw[line width=2mm] (0.6469,0.0000) -- (0.6665,0.0000);
\draw[line width=2mm] (0.7005,0.0000) -- (0.7180,0.0000);
\draw[line width=2mm] (0.7339,0.0000) -- (0.7488,0.0000);
\draw[line width=2mm] (0.8537,0.0000) -- (0.8734,0.0000);
\draw[line width=2mm] (0.8884,0.0000) -- (0.9079,0.0000);
\draw[line width=2mm] (0.9495,0.0000) -- (0.9676,0.0000);
\draw[line width=2mm] (0.9874,0.0000) -- (1.0055,0.0000);
\draw[line width=2mm] (2.0260,0.0000) -- (2.0408,0.0000);
\draw[line width=2mm] (2.0507,0.0000) -- (2.0636,0.0000);
\draw[line width=2mm] (2.1037,0.0000) -- (2.1168,0.0000);
\draw[line width=2mm] (2.1260,0.0000) -- (2.1376,0.0000);
\draw[line width=2mm] (2.2452,0.0000) -- (2.2623,0.0000);
\draw[line width=2mm] (2.2739,0.0000) -- (2.2900,0.0000);
\draw[line width=2mm] (2.3164,0.0000) -- (2.3327,0.0000);
\draw[line width=2mm] (2.3441,0.0000) -- (2.3590,0.0000);
\draw[line width=2mm] (2.7087,0.0000) -- (2.7223,0.0000);
\draw[line width=2mm] (2.7346,0.0000) -- (2.7482,0.0000);
\draw[line width=2mm] (2.7860,0.0000) -- (2.8003,0.0000);
\draw[line width=2mm] (2.8136,0.0000) -- (2.8304,0.0000);
\draw[line width=2mm] (2.9533,0.0000) -- (2.9694,0.0000);
\draw[line width=2mm] (2.9820,0.0000) -- (2.9957,0.0000);
\draw[line width=2mm] (3.0403,0.0000) -- (3.0545,0.0000);
\draw[line width=2mm] (3.0669,0.0000) -- (3.0844,0.0000);
\draw[line width=2mm] (6.6002,0.0000) -- (6.6186,0.0000);
\draw[line width=2mm] (6.6364,0.0000) -- (6.6598,0.0000);
\draw[line width=2mm] (6.7069,0.0000) -- (6.7250,0.0000);
\draw[line width=2mm] (6.7377,0.0000) -- (6.7529,0.0000);
\draw[line width=2mm] (6.8853,0.0000) -- (6.9003,0.0000);
\draw[line width=2mm] (6.9135,0.0000) -- (6.9281,0.0000);
\draw[line width=2mm] (6.9768,0.0000) -- (6.9914,0.0000);
\draw[line width=2mm] (7.0051,0.0000) -- (7.0219,0.0000);
\draw[line width=2mm] (7.4297,0.0000) -- (7.4462,0.0000);
\draw[line width=2mm] (7.4587,0.0000) -- (7.4779,0.0000);
\draw[line width=2mm] (7.5189,0.0000) -- (7.5379,0.0000);
\draw[line width=2mm] (7.5546,0.0000) -- (7.5719,0.0000);
\draw[line width=2mm] (7.7097,0.0000) -- (7.7286,0.0000);
\draw[line width=2mm] (7.7448,0.0000) -- (7.7655,0.0000);
\draw[line width=2mm] (7.8341,0.0000) -- (7.8525,0.0000);
\draw[line width=2mm] (7.8721,0.0000) -- (7.8922,0.0000);
\draw[line width=2mm] (8.6705,0.0000) -- (8.6877,0.0000);
\draw[line width=2mm] (8.7015,0.0000) -- (8.7160,0.0000);
\draw[line width=2mm] (8.7705,0.0000) -- (8.7870,0.0000);
\draw[line width=2mm] (8.8007,0.0000) -- (8.8170,0.0000);
\draw[line width=2mm] (8.9603,0.0000) -- (8.9794,0.0000);
\draw[line width=2mm] (8.9942,0.0000) -- (9.0106,0.0000);
\draw[line width=2mm] (9.0469,0.0000) -- (9.0598,0.0000);
\draw[line width=2mm] (9.0729,0.0000) -- (9.0874,0.0000);
\draw[line width=2mm] (9.5819,0.0000) -- (9.5966,0.0000);
\draw[line width=2mm] (9.6125,0.0000) -- (9.6263,0.0000);
\draw[line width=2mm] (9.6674,0.0000) -- (9.6823,0.0000);
\draw[line width=2mm] (9.6988,0.0000) -- (9.7125,0.0000);
\draw[line width=2mm] (9.8541,0.0000) -- (9.8714,0.0000);
\draw[line width=2mm] (9.8878,0.0000) -- (9.9030,0.0000);
\draw[line width=2mm] (9.9471,0.0000) -- (9.9653,0.0000);
\draw[line width=2mm] (9.9812,0.0000) -- (10.0000,0.0000);

\node at (3.9399,-0.4000) {$I$};
\draw[line width=2mm,color=lightgray] (0.8884,-0.7000) -- (6.9914,-0.7000);
\draw[line width=2mm] (0.0000,-0.7000) -- (0.0156,-0.7000);
\draw[line width=2mm] (0.0275,-0.7000) -- (0.0411,-0.7000);
\draw[line width=2mm] (0.0762,-0.7000) -- (0.0919,-0.7000);
\draw[line width=2mm] (0.1077,-0.7000) -- (0.1233,-0.7000);
\draw[line width=2mm] (0.2181,-0.7000) -- (0.2291,-0.7000);
\draw[line width=2mm] (0.2393,-0.7000) -- (0.2496,-0.7000);
\draw[line width=2mm] (0.2838,-0.7000) -- (0.2963,-0.7000);
\draw[line width=2mm] (0.3057,-0.7000) -- (0.3199,-0.7000);
\draw[line width=2mm] (0.6164,-0.7000) -- (0.6339,-0.7000);
\draw[line width=2mm] (0.6469,-0.7000) -- (0.6665,-0.7000);
\draw[line width=2mm] (0.7005,-0.7000) -- (0.7180,-0.7000);
\draw[line width=2mm] (0.7339,-0.7000) -- (0.7488,-0.7000);
\draw[line width=2mm] (0.8537,-0.7000) -- (0.8734,-0.7000);
\draw[line width=2mm] (0.8884,-0.7000) -- (0.9079,-0.7000);
\draw[line width=2mm] (0.9495,-0.7000) -- (0.9676,-0.7000);
\draw[line width=2mm] (0.9874,-0.7000) -- (1.0055,-0.7000);
\draw[line width=2mm] (2.0260,-0.7000) -- (2.0408,-0.7000);
\draw[line width=2mm] (2.0507,-0.7000) -- (2.0636,-0.7000);
\draw[line width=2mm] (2.1037,-0.7000) -- (2.1168,-0.7000);
\draw[line width=2mm] (2.1260,-0.7000) -- (2.1376,-0.7000);
\draw[line width=2mm] (2.2452,-0.7000) -- (2.2623,-0.7000);
\draw[line width=2mm] (2.2739,-0.7000) -- (2.2900,-0.7000);
\draw[line width=2mm] (2.3164,-0.7000) -- (2.3327,-0.7000);
\draw[line width=2mm] (2.3441,-0.7000) -- (2.3590,-0.7000);
\draw[line width=2mm] (2.7087,-0.7000) -- (2.7223,-0.7000);
\draw[line width=2mm] (2.7346,-0.7000) -- (2.7482,-0.7000);
\draw[line width=2mm] (2.7860,-0.7000) -- (2.8003,-0.7000);
\draw[line width=2mm] (2.8136,-0.7000) -- (2.8304,-0.7000);
\draw[line width=2mm] (2.9533,-0.7000) -- (2.9694,-0.7000);
\draw[line width=2mm] (2.9820,-0.7000) -- (2.9957,-0.7000);
\draw[line width=2mm] (3.0403,-0.7000) -- (3.0545,-0.7000);
\draw[line width=2mm] (3.0669,-0.7000) -- (3.0844,-0.7000);
\draw[line width=2mm] (6.6002,-0.7000) -- (6.6186,-0.7000);
\draw[line width=2mm] (6.6364,-0.7000) -- (6.6598,-0.7000);
\draw[line width=2mm] (6.7069,-0.7000) -- (6.7250,-0.7000);
\draw[line width=2mm] (6.7377,-0.7000) -- (6.7529,-0.7000);
\draw[line width=2mm] (6.8853,-0.7000) -- (6.9003,-0.7000);
\draw[line width=2mm] (6.9135,-0.7000) -- (6.9281,-0.7000);
\draw[line width=2mm] (6.9768,-0.7000) -- (6.9914,-0.7000);
\draw[line width=2mm] (7.0051,-0.7000) -- (7.0219,-0.7000);
\draw[line width=2mm] (7.4297,-0.7000) -- (7.4462,-0.7000);
\draw[line width=2mm] (7.4587,-0.7000) -- (7.4779,-0.7000);
\draw[line width=2mm] (7.5189,-0.7000) -- (7.5379,-0.7000);
\draw[line width=2mm] (7.5546,-0.7000) -- (7.5719,-0.7000);
\draw[line width=2mm] (7.7097,-0.7000) -- (7.7286,-0.7000);
\draw[line width=2mm] (7.7448,-0.7000) -- (7.7655,-0.7000);
\draw[line width=2mm] (7.8341,-0.7000) -- (7.8525,-0.7000);
\draw[line width=2mm] (7.8721,-0.7000) -- (7.8922,-0.7000);
\draw[line width=2mm] (8.6705,-0.7000) -- (8.6877,-0.7000);
\draw[line width=2mm] (8.7015,-0.7000) -- (8.7160,-0.7000);
\draw[line width=2mm] (8.7705,-0.7000) -- (8.7870,-0.7000);
\draw[line width=2mm] (8.8007,-0.7000) -- (8.8170,-0.7000);
\draw[line width=2mm] (8.9603,-0.7000) -- (8.9794,-0.7000);
\draw[line width=2mm] (8.9942,-0.7000) -- (9.0106,-0.7000);
\draw[line width=2mm] (9.0469,-0.7000) -- (9.0598,-0.7000);
\draw[line width=2mm] (9.0729,-0.7000) -- (9.0874,-0.7000);
\draw[line width=2mm] (9.5819,-0.7000) -- (9.5966,-0.7000);
\draw[line width=2mm] (9.6125,-0.7000) -- (9.6263,-0.7000);
\draw[line width=2mm] (9.6674,-0.7000) -- (9.6823,-0.7000);
\draw[line width=2mm] (9.6988,-0.7000) -- (9.7125,-0.7000);
\draw[line width=2mm] (9.8541,-0.7000) -- (9.8714,-0.7000);
\draw[line width=2mm] (9.8878,-0.7000) -- (9.9030,-0.7000);
\draw[line width=2mm] (9.9471,-0.7000) -- (9.9653,-0.7000);
\draw[line width=2mm] (9.9812,-0.7000) -- (10.0000,-0.7000);
\draw[line width=2mm,color=white] (-0.1000,-0.7000) -- (0.8384,-0.7000);
\draw[line width=2mm,color=white] (7.0414,-0.7000) -- (10.1000,-0.7000);

\node at (5,1.0000) {$I_{j}$};
\draw[line width=2mm,color=lightgray] (0,0.7000) -- (10.0000,0.7000);
\draw[line width=2mm] (0.0000,0.7000) -- (0.0156,0.7000);
\draw[line width=2mm] (0.0275,0.7000) -- (0.0411,0.7000);
\draw[line width=2mm] (0.0762,0.7000) -- (0.0919,0.7000);
\draw[line width=2mm] (0.1077,0.7000) -- (0.1233,0.7000);
\draw[line width=2mm] (0.2181,0.7000) -- (0.2291,0.7000);
\draw[line width=2mm] (0.2393,0.7000) -- (0.2496,0.7000);
\draw[line width=2mm] (0.2838,0.7000) -- (0.2963,0.7000);
\draw[line width=2mm] (0.3057,0.7000) -- (0.3199,0.7000);
\draw[line width=2mm] (0.6164,0.7000) -- (0.6339,0.7000);
\draw[line width=2mm] (0.6469,0.7000) -- (0.6665,0.7000);
\draw[line width=2mm] (0.7005,0.7000) -- (0.7180,0.7000);
\draw[line width=2mm] (0.7339,0.7000) -- (0.7488,0.7000);
\draw[line width=2mm] (0.8537,0.7000) -- (0.8734,0.7000);
\draw[line width=2mm] (0.8884,0.7000) -- (0.9079,0.7000);
\draw[line width=2mm] (0.9495,0.7000) -- (0.9676,0.7000);
\draw[line width=2mm] (0.9874,0.7000) -- (1.0055,0.7000);
\draw[line width=2mm] (2.0260,0.7000) -- (2.0408,0.7000);
\draw[line width=2mm] (2.0507,0.7000) -- (2.0636,0.7000);
\draw[line width=2mm] (2.1037,0.7000) -- (2.1168,0.7000);
\draw[line width=2mm] (2.1260,0.7000) -- (2.1376,0.7000);
\draw[line width=2mm] (2.2452,0.7000) -- (2.2623,0.7000);
\draw[line width=2mm] (2.2739,0.7000) -- (2.2900,0.7000);
\draw[line width=2mm] (2.3164,0.7000) -- (2.3327,0.7000);
\draw[line width=2mm] (2.3441,0.7000) -- (2.3590,0.7000);
\draw[line width=2mm] (2.7087,0.7000) -- (2.7223,0.7000);
\draw[line width=2mm] (2.7346,0.7000) -- (2.7482,0.7000);
\draw[line width=2mm] (2.7860,0.7000) -- (2.8003,0.7000);
\draw[line width=2mm] (2.8136,0.7000) -- (2.8304,0.7000);
\draw[line width=2mm] (2.9533,0.7000) -- (2.9694,0.7000);
\draw[line width=2mm] (2.9820,0.7000) -- (2.9957,0.7000);
\draw[line width=2mm] (3.0403,0.7000) -- (3.0545,0.7000);
\draw[line width=2mm] (3.0669,0.7000) -- (3.0844,0.7000);
\draw[line width=2mm] (6.6002,0.7000) -- (6.6186,0.7000);
\draw[line width=2mm] (6.6364,0.7000) -- (6.6598,0.7000);
\draw[line width=2mm] (6.7069,0.7000) -- (6.7250,0.7000);
\draw[line width=2mm] (6.7377,0.7000) -- (6.7529,0.7000);
\draw[line width=2mm] (6.8853,0.7000) -- (6.9003,0.7000);
\draw[line width=2mm] (6.9135,0.7000) -- (6.9281,0.7000);
\draw[line width=2mm] (6.9768,0.7000) -- (6.9914,0.7000);
\draw[line width=2mm] (7.0051,0.7000) -- (7.0219,0.7000);
\draw[line width=2mm] (7.4297,0.7000) -- (7.4462,0.7000);
\draw[line width=2mm] (7.4587,0.7000) -- (7.4779,0.7000);
\draw[line width=2mm] (7.5189,0.7000) -- (7.5379,0.7000);
\draw[line width=2mm] (7.5546,0.7000) -- (7.5719,0.7000);
\draw[line width=2mm] (7.7097,0.7000) -- (7.7286,0.7000);
\draw[line width=2mm] (7.7448,0.7000) -- (7.7655,0.7000);
\draw[line width=2mm] (7.8341,0.7000) -- (7.8525,0.7000);
\draw[line width=2mm] (7.8721,0.7000) -- (7.8922,0.7000);
\draw[line width=2mm] (8.6705,0.7000) -- (8.6877,0.7000);
\draw[line width=2mm] (8.7015,0.7000) -- (8.7160,0.7000);
\draw[line width=2mm] (8.7705,0.7000) -- (8.7870,0.7000);
\draw[line width=2mm] (8.8007,0.7000) -- (8.8170,0.7000);
\draw[line width=2mm] (8.9603,0.7000) -- (8.9794,0.7000);
\draw[line width=2mm] (8.9942,0.7000) -- (9.0106,0.7000);
\draw[line width=2mm] (9.0469,0.7000) -- (9.0598,0.7000);
\draw[line width=2mm] (9.0729,0.7000) -- (9.0874,0.7000);
\draw[line width=2mm] (9.5819,0.7000) -- (9.5966,0.7000);
\draw[line width=2mm] (9.6125,0.7000) -- (9.6263,0.7000);
\draw[line width=2mm] (9.6674,0.7000) -- (9.6823,0.7000);
\draw[line width=2mm] (9.6988,0.7000) -- (9.7125,0.7000);
\draw[line width=2mm] (9.8541,0.7000) -- (9.8714,0.7000);
\draw[line width=2mm] (9.8878,0.7000) -- (9.9030,0.7000);
\draw[line width=2mm] (9.9471,0.7000) -- (9.9653,0.7000);
\draw[line width=2mm] (9.9812,0.7000) -- (10.0000,0.7000);

\draw[line width=2mm,color=gray] (0.8884,-1.4000) -- (3.0844,-1.4000);
\node at (1.9864,-1.1000) {$J_0$};
\draw[line width=2mm,color=gray] (6.6002,-1.4000) -- (6.9914,-1.4000);
\node at (6.7958,-1.1000) {$J_1$};
\draw [decorate,decoration={brace,amplitude=6pt,mirror,raise=2pt},yshift=0pt](0.0000,-1.5000) -- (0.8884,-1.5000) node [black,midway,yshift=-0.6cm] {\footnotesize $2t_*$};\draw [decorate,decoration={brace,amplitude=6pt,mirror,raise=2pt},yshift=0pt](3.0844,-1.5000) -- (6.6002,-1.5000) node [black,midway,yshift=-0.6cm] {\footnotesize $G_j$};
\end{tikzpicture}

    \caption{The image above shows some of the notations used in the proof. The black parts inside the light grey intervals are some of the basic intervals of the Cantor set. Note that the endpoints of \( I \) coincide with the endpoints of basic intervals and also that \( I \) must be contained in \( I_j \) since if it was not, \( G_j \) would not be the oldest gap in \( I \). Note also that the right endpoint of \( I_{j0} \) and \( J_0 \) coincide.}
  \end{figure}
  Let \( w \) be the midpoint of \( J_0 \) and \( \delta = |J_0| \) and consider the function
  \begin{equation*}
    f(t_0,t_1) =  h(w-t_0+t_1, \delta+2t_0+2t_1).
  \end{equation*}
  As $h$ is increasing as an interval function, by the definition of $f(t_0,t_1)$ we have $\frac{\partial f}{\partial t_0} \geq 0$ and $\frac{\partial f}{\partial t_1} \geq 0$. Also, by the third assumption, for sufficiently small \( t_0 \) and \( t_1 \), 
  \begin{equation*}
    \frac{\partial^2 f}{\partial t_0 \partial t_1}(t_0,t_1) = (-h_{11} + 4h_{22})|_{w-t_0+t_1, \delta+2t_0+2t_1} \overset{(\ref{eq:h1})}{\leq} 0
  \end{equation*}
  and
  \begin{equation*}
    \frac{\partial^2 f}{\partial t_1^2}(t_0,t_1) = (h_{11} + 4h_{12}+ 4h_{22})|_{w-t_0+t_1, \delta+2t_0+2t_1} \overset{(\ref{eq:h2})}{\leq} 0.
  \end{equation*}

  Since \( J_0  \) and \( I_{j0} \) have their right endpoint in common and \( J_0 \subseteq I_{j0} \) there exists a unique number \( t_* \in \mathbb{R}_+ \) such that \( I(w-t_*, \delta+2t_*) = I_{j0} \). Set \( f(t) = f(0,\frac{t}{2}) \) and \( f_*(t) = f(t_*,\frac{t}{2}) \). Then
  \begin{equation*}
    f_*(t) = f(t_*,\frac{t}{2}) = h(w - t_* + \frac{t}{2}, \delta + 2t_* + t) = h(w_{j0} + \frac{t}{2}, |I_{j0}| + t)
  \end{equation*}
  which implies 
  \begin{equation}
    f_*(0) = h(I_{j0}) = h(I_{j1})
    \label{eq:T1}
  \end{equation}
  and 
  \begin{equation}
    f_*(|G_j \cup I_{j1}|) = h(I_j) .
    \label{eq:T2}
  \end{equation}
  Similarly, 
  \begin{equation}
    f(0) = h(J_0) .
    \label{eq:fh}
  \end{equation}
  As \( \frac{ \partial^2 f}{\partial t_0 \partial t_1}  \leq 0 \),  \( \frac{\partial}{\partial t_1} f(t_0,t_1) \) decreases as \( t_0 \) increases for all \( t_1 \). This implies
  \begin{equation}
    f_*'(t)\leq f'(t) 
    \label{eq:ff}
  \end{equation}
  for all \( t \) which in turn implies \( f(t)-f(0) \geq f_*(t) - f_*(0) \) for all \( t \).

  Set \( T = |G_j \cup I_{j1}| \). Then
  \begin{equation}
    \begin{split}
      \hspace{-1em}f_*(T) - f_*(0)
    \overset{(\ref{eq:T1},\ref{eq:T2})}{=}&  h(I_j) - h(I_{j0})  
      \overset{(\ref{eq:qrinequality})}{\geq} q \cdot \nu(I_j) - r \cdot \nu(I_{j0}) 
      \\=&
      \,q \cdot \bigl( \nu (I_{j} ) - \nu(I_{j0})\bigr) - (r-q) \cdot \nu(I_{j0})  
      \\
      =& \,q \cdot \nu(I_{j1}) - (r-q) \cdot \nu(I_{j1})= 
      \bigl( q - (r-q)\bigr) \cdot \nu(I_{j1}).
    \end{split}
    \label{eq:eq}
  \end{equation}
  Since \( \frac{\partial^2f}{\partial t_1^2}\leq 0 \) and \( \frac{\partial f}{\partial t_1} \geq 0\), \( f_*'(t) \) is positive and decreasing. Using this we obtain
  \begin{equation}
    \begin{split}
      f(\rho T)-f(0) &= \int_0^{\rho  T} f'(t) \,  dt  \overset{(\ref{eq:ff})}{\geq}  \int_0^{\rho  T} f_*'(t) \,  dt  \\
      &\hspace{-2em}\overset{f_*' decreasing}{\geq}  \rho  \cdot \bigl( f_*(T) - f_*(0) \bigr) \overset{(\ref{eq:eq})}{\geq}  \bigl( q - (r - q)\bigr) \cdot \rho  \cdot \nu (I_{j1})
    \end{split}
    \label{eq:GOOD}
  \end{equation}
  for all \( \rho \in [0,1] \). Now fix \( \rho \in [0,1] \) as the unique number such that \( \rho T = |G_j \cup J_1| \), i.e.\ set \( \rho = \frac{|G_j \cup J_1|}{|G_j \cup I_{j1}|} \). Then 
  \begin{equation}
    \rho \cdot_L  (G_j \cup I_{j1}) = G_j \cup J_1.
    \label{eq:noname}
  \end{equation}
  By lemma~\ref{lem:staircase} and the second assumption, we have 
  \begin{equation}
    \rho \cdot \nu (I_{j1}) \geq \nu (\rho \cdot_L (G_j \cup I_{j1})).
    \label{eq:goodgoodeq2}
  \end{equation}
  Using this inequality and the previous equations  we get
  \begin{equation*}
    \begin{split}
      h (I) 
      = f(\rho) 
      &\overset{(\ref{eq:GOOD})}{\geq} f(0) +\bigl( q - (r - q)\bigr)\cdot \rho  \cdot \nu (I_{j1})  
      \\
      &\overset{(\ref{eq:fh})}{=} h(J_0) + \bigl( q - (r - q)\bigr)\cdot \rho  \cdot \nu (I_{j1}) 
      \\
      &\overset{(\ref{eq:goodgoodeq2})}{\geq} h(J_0) + \bigl( q - (r - q)\bigr) \cdot \nu (\rho \cdot_L  (G_j \cup I_{j1})) 
      \\
      &\overset{(\ref{eq:noname})}{=}  
      h(J_0) + \bigl( q - (r - q)\bigr) \cdot \nu (G_j \cup J_1) 
      \\
      &=h(J_0) + \bigl( q - (r - q)\bigr) \cdot \nu (J_1).
    \end{split}
  \end{equation*}

  As  we can repeat this procedure with \( J_0 \) instead of \( I \) arbitrarily many times and \( { h(w,\delta) \to 0} \) as \( \delta \to 0 \) for all $w$ we can conclude that
  \begin{equation*}
    h (I) \geq  \bigl( q - (r - q)\bigr) \cdot \nu (I).
  \end{equation*}
  This proves the theorem.

\end{proof}

\begin{remark} \label{rem:reverse}
  The symmetric theorem also holds, i.e.\ we can assume \(h_{11}-4h_{12}+4h_{22} \leq 0 \) and \( \frac{1}{2^m} \leq \frac{|G_j \cup I_{j01^m}|}{|G_j \cup I_{j0}|} \) instead of assuming \( h_{11}+4h_{12}+4h_{22} \leq 0 \) and \( \frac{1}{2^m} \leq \frac{|G_j \cup I_{j10^m}|}{|G_j \cup I_{j1}|} \).
\end{remark}

\section{Examples.}

We will end this paper with three examples which use Theorem~\ref{thm:general-1-thm} to calculate the exact Hausdorff measure of three Cantor sets studied in \cite{smooth_cantor_sets} and \cite{pCantorset_measure}.

\begin{example} \label{ex:1}
  In~\cite{pCantorset_measure},  Cabrielli, Molter, Paulauskas and Shonkwiler studied the Cantor sets \( C_p\) associated with the sequence of gap lengths 
  \begin{equation*}
    |G_l^k| = \frac{1}{(2^k + l)^p}
  \end{equation*}
  where \( p \) is any real number which is strictly larger than one. For the Hausdorff measure $\mu_h$ associated to the gauge function \( h(w, \delta) = \delta^{1/p} \) the following bounds were acquired (see~\cite{pCantorset_measure}, theorem~1.1)
  \begin{equation*}
    \frac{1}{8} \left( \frac{2^p}{2^p-2} \right)^{1/p} \!\!\!  \leq \mu_h(C_{p}) \leq \left( \frac{1}{p-1} \right)^{1/p}
  \end{equation*}

  We will show that  by using Corollary~\ref{cor:the_cor}, we can compute the exact value of \( \mu_h(C_{p}) \) for any \( p>1 \). As \( \delta^{1/p} \) is concave for any fixed \( p>1 \) and \( \{ G_l^k \} \) is a decreasing gap sequence, we only need to find good estimates of \( q \) and \( r \). To find such estimates we will need the following result from~\cite{pCantorset_measure}, which gives bounds for the length of the basic intervals associated to the Cantor sets considered.
  \begin{equation} \label{eq:lengthineq}
    \frac{2^p}{2^p-2} \cdot \left( \frac{1}{2^k+l+1} \right)^p \leq |I_l^k| \leq \frac{2^p}{2^p-2} \cdot \left( \frac{1}{2^k+l} \right)^p.
  \end{equation}

  We will now calculate estimates for the constants \(r\) and \(q\) in equation~\ref{eq:qrinequality}. To this end, note that if \( I_l^k \) and \( I_{l'}^{k'}  \) are any two  basic intervals associated with \( C_p \) with \( I_{l'}^{k'} \subseteq I_l^k \) we have
\begin{equation}
\frac{l'}{2^{k'}} \geq \frac{l}{2^k}
\label{eq:ll'relation}
\end{equation}
  and also, by the definition of the Cantor measure
\begin{equation}
\nu(I_{l'}^{k'}) = \frac{1}{2^{k'}} \quad \textrm{and} \quad \nu(I_{l}^{k}) = \frac{1}{2^{k}}.
\label{eq:cmeasure}
\end{equation}
  Using these observations, we get
\begin{equation*}
\begin{split}
h(I_{l'}^{k'}) &= |I_{l'}^{k'}|^{1/p} 
\\&\!\!\overset{(\ref{eq:lengthineq})}{\leq}
 \frac{2}{(2^p-2)^{1/p}} \cdot \frac{1}{2^{k'}+l'}  
\\&=
\frac{2}{(2^p-2)^{1/p}} \cdot \frac{1}{1+\frac{l'}{2^{k'}}} \cdot \frac{1}{2^{k'}}  
\\&\!\!\overset{(\ref{eq:cmeasure})}{=}
 \frac{2}{(2^p-2)^{1/p}} \cdot \frac{1}{1+\frac{l'}{2^{k'}}} \cdot \nu\left(I_{l'}^{k'}\right) 
\\&\!\!\overset{(\ref{eq:ll'relation})}{\leq} 
 \frac{2}{(2^p-2)^{1/p}} \cdot \frac{1}{1+\frac{l}{2^{k}}} \cdot \nu\left(I_{l'}^{k'}\right).
\end{split}
\end{equation*}
  Completely analogously, we get the lower limit
\begin{equation*}
 \frac{2}{(2^p-2)^{1/p}} \cdot \frac{1}{1+\frac{l+1}{2^{k}}} \cdot \nu(I_{l'}^{k'})  \leq h(I_{l'}^{k'}).
\end{equation*}
  Combining the upper and lower limit we obtain the following estimates of \(q\) and \(r\) for all basic intervals contained in \(I_l^k\).
\begin{equation}
 \frac{2}{(2^p-2)^{1/p}} \cdot \frac{1}{1+\frac{l+1}{2^{k}}} \cdot \nu(I_{l'}^{k'})  \leq h(I_{l'}^{k'}) \leq  \frac{2}{(2^p-2)^{1/p}} \cdot \frac{1}{1+\frac{l}{2^{k}}} \cdot \nu(I_{l'}^{k'}).
\label{eq:nu-h-inequalities}
\end{equation}
  This yields
\begin{equation*}
\begin{split}
\mu_h (C_p) =& \sum_{l=0}^{2^k-1}  m_h (C_p \cap I_l^k)  
\\
\overset{(\ref{eq:nu-h-inequalities})}{\leq} &
\sum_{l=0}^{2^k-1}  \frac{2}{(2^p-2)^{1/p}} \cdot \frac{1}{1 + \frac{l}{2^k}} \cdot \nu(I_l^k) 
\\
\overset{(\ref{eq:cmeasure})}{=}&
\frac{2}{(2^p-2)^{1/p}} \cdot \sum_{l=0}^{2^k-1}   \frac{1}{1 + \frac{l}{2^k}} \cdot \frac{1}{2^k}.
\end{split}
\end{equation*}
  Since this is true for all \( k \), and
\begin{equation*}
\lim_{k \to \infty} \sum_{l=0}^{2^k-1}   \frac{1}{1 + \frac{l}{2^k}} \cdot \frac{1}{2^k} = \int_0^1 \frac{dx}{1+x} = \left[\log (1+x)\right]_0^1 = \log 2 
\end{equation*}
  we get
\begin{equation}
m_{1/p} (C_p) \leq \frac{2 \log 2}{(2^p-2)^{1/p}}.
\label{eq:upper_p-limit}
\end{equation}
 Similarly for the lower limit;
\begin{equation*}
\begin{split}
m_{1/p} (C_p) =& \sum_{l=0}^{2^k-1}  m_{1/p} \left(C_p \cap I_l^k\right) 
\\
\overset{(\ref{eq:nu-h-inequalities})}{\geq}&
 \sum_{l=0}^{2^k-1}  \frac{2}{(2^p-2)^{1/p}} \cdot \left( \frac{1}{1 + \frac{l+1}{2^k}}  -    \left(  \frac{1}{1 + \frac{l}{2^k}}  - \frac{1}{1 + \frac{l+1}{2^k}}   \right) \right)  \cdot \nu(I_l^k) 
\\
\geq&
 \frac{2}{(2^p-2)^{1/p}} \cdot \sum_{l=0}^{2^k-1}   \left( \frac{1}{1 + \frac{l+1}{2^k}}  -   \frac{1}{2^k}    \right)  \cdot \nu(I_l^k)  
\\
\overset{(\ref{eq:cmeasure})}{=} &
\frac{2}{(2^p-2)^{1/p}} \cdot \sum_{l=0}^{2^k-1}   \left( \frac{1}{1 + \frac{l+1}{2^k}}  -   \frac{1}{2^k}    \right)  \cdot \frac{1}{2^k} 
\\
=&
\frac{2}{(2^p-2)^{1/p}} \cdot \sum_{l=0}^{2^k-1}   \frac{1}{1 + \frac{l+1}{2^k}}   \cdot \frac{1}{2^k} 
-
\frac{2}{(2^p-2)^{1/p}} \cdot   \frac{1}{2^k}.
\end{split}
\end{equation*}
  As 
\begin{equation*}
\begin{split}
\lim_{k \to \infty} 
\frac{2}{(2^p-2)^{1/p}} \cdot \sum_{l=0}^{2^k-1}   \frac{1}{1 + \frac{l+1}{2^k}}   \cdot \frac{1}{2^k} 
-
\frac{2}{(2^p-2)^{1/p}} \cdot   \frac{1}{2^k}   
\\
= \displaystyle \frac{2}{(2^p-2)^{1/p}} \cdot \int_0^1 \frac{1}{1+x}dx = \frac{2 \log 2}{(2^p-2)^{1/p}}
\end{split}
\end{equation*}
  we get the the lower limit
\begin{equation}
m_{1/p} (C_p) \geq \frac{2 \log 2}{(2^p-2)^{1/p}}.
\label{eq:lower_p-limit}
\end{equation}
  By combining equation~\ref{eq:upper_p-limit} and equation~\ref{eq:lower_p-limit} we can conclude that 
\begin{equation*}
m_{1/p} (C_p) =\frac{2 \log 2}{(2^p-2)^{1/p}}.
\end{equation*}
\end{example}

\begin{example} \label{ex:2}
  As a small variation of the Cantor sets studied in the previous example we can consider the Cantor sets \( C_{p,x} \) associated with the sequences of gap lengths 
\begin{equation}
|G_l^k| = \frac{1}{([x^k] + l)^p} 
\label{eq:gdef2}
\end{equation}
  where \( p>1 \) and \( x > 2 \). These sets were also studied in~\cite{pCantorset_measure} where Cabrielli, Molter, Mendevil, Paulauskas and Shonkwiler gave the bounds 
\begin{equation*}
0<m_{\frac{\log 2}{p \log x}}(C_{p,x}) \leq \left( \frac{4^p}{2^p-2}\right)^{\frac{\log 2}{p \log x}}
\end{equation*}
 As in the previous example, we will calculate the measure of these Cantor sets using Corollary~\ref{cor:the_cor}, for the gauge function \( h_{p,x}(w, \delta) = \delta^{\frac{\log 2}{p \log x}}\), for any fixed \( p>1 \) and \(x>2\), where \( x \) and \( p \) are the parameters for the Cantor set considered. As the gauge function is clearly concave and the gap sequence is decreasing, we only need to find estimates for \( q \) and \( r \).

We begin by calculating upper and lower bounds for \(|I_l^k|\) similar to those in equation~\ref{eq:lengthineq}.
\begin{equation}
\begin{split}
|I_l^k| 
=&
\sum_{h = 0}^\infty \sum_{j=0}^{2^h-1} \left|G^{k+h}_{2^hl+j}\right| 
\overset{(\ref{eq:gdef2})}{\leq}
\sum_{h = 0}^\infty \sum_{j=0}^{2^h-1} \left|G^{k+h}_{2^hl+0}\right| 
\\
\overset{(\ref{eq:gdef2})}{=}&
 \sum_0^\infty \frac{2^h}{(\lfloor x^{k+h}\rfloor  + l \cdot 2^h)^p} 
\leq
 \frac{1}{x^{kp}} \cdot \sum_{h=0}^{\infty} \frac{2^h}{\bigl( x^h - \frac{1}{x^k} \bigr)^p} 
\\  
\leq &
\frac{1}{x^{kp}} \cdot (1 + \varepsilon^{(1)}_k ) \cdot \sum_{h=0}^{\infty} \frac{2^h}{x^{ph}} = \frac{1}{x^{kp}} \cdot \left(1 + \varepsilon^{(1)}_k \right) \cdot \frac{1}{1 - \frac{2}{x^p}}
\end{split}
\label{eq:ex2upperbound}
\end{equation}
  where \( \varepsilon^{(1)}_k \) is a small positive number which tends to zero as \( k \to \infty \). Similarly, but by somewhat more tedious calculations, we obtain
\begin{equation}
\begin{split}
|I_l^k| =&  \sum_{h = 0}^\infty \sum_{j=0}^{2^h-1} \left|G^{k+h}_{2^hl+j}\right| \geq \sum_{h=0}^\infty \frac{2^h}{\left(\lfloor x^{k+h} \rfloor + l \cdot 2^h + 2^h - 1\right)^p} 
\\
\geq& \sum_{h=0}^\infty \frac{2^h}{(x^{k+h} + (l+1) \cdot 2^h )^p} 
\geq \frac{1}{x^{kp}} \cdot \sum_{h=0}^\infty \frac{2^h}{x^{hp} \cdot \left(1 + \frac{l+1}{x^k} \cdot \frac{2^h}{x^h}\right)^p} 
\\
\geq& 
\frac{1}{x^{kp}} \cdot \sum_{h=0}^\infty \frac{2^h}{x^{hp} \cdot \left(1 +  \frac{2^k}{x^k} \cdot \frac{2^h}{x^h}\right)^p}  
\geq \frac{1}{x^{kp}} \cdot \sum_{h=0}^\infty \frac{2^h}{x^{hp}} \cdot  \frac{1}{\left(1 +  \frac{2^k}{x^k} \right)^p} \\
\geq& \frac{1}{x^{kp}} \cdot \frac{1}{1 - \frac{2}{x^p}} \cdot \frac{1}{\left(1 + \frac{2^k}{x^k}\right)^p} = \frac{1}{x^{kp}} \cdot \frac{1}{1 - \frac{2}{x^p}} \cdot \left(1 - \varepsilon^{(0)}_k \right)
\end{split}
\label{eq:ex2lowerbound}
\end{equation}
  where \( \varepsilon^{(0)}_k \) is a small positive number which tends to zero as \( k \to \infty \). Summing up equations~\ref{eq:ex2upperbound}~and~\ref{eq:ex2lowerbound}, we have
\begin{equation*}
\frac{1}{x^{kp}} \cdot \frac{1}{1 - \frac{2}{x^p}} \cdot \left(1 - \varepsilon^{(0)}_k \right) 
\leq 
|I_l^k| 
\leq
 \frac{1}{x^{kp}} \cdot \left(1 + \varepsilon^{(1)}_k \right) \cdot \frac{1}{1 - \frac{2}{x^p}}.
\end{equation*}
  This implies
\begin{equation*}
\begin{split}
 \left( \frac{1}{1 - \frac{2}{x^p}} \right)^\frac{\log 2}{p \log x}  \cdot \left(1 - \varepsilon^{(0)}_k \right)^\frac{\log 2}{p \log x} \cdot \frac{1}{2^k} 
\leq
|I_l^k|^\frac{\log 2}{p \log x} 
\\
\leq
 \frac{1}{x^{kp}} \cdot \left(1 + \varepsilon^{(1)}_k \right)^\frac{\log 2}{p \log x} \cdot \left( \frac{1}{1 - \frac{2}{x^p}} \right)^\frac{\log 2}{p \log x}  \cdot  \frac{1}{2^k}
\end{split}
\end{equation*}
  which, by the definition of $h_{p,x}$ and \(\nu\) can be written as 
\begin{equation*}
\begin{split}
 \left( \frac{1}{1 - \frac{2}{x^p}} \right)^\frac{\log 2}{p \log x}  \cdot \left(1 - \varepsilon^{(0)}_k \right)^\frac{\log 2}{p \log x} \cdot \nu (I_l^k)
\leq
h_{p,x} (I_l^k)
\\
\leq
 \left(1 + \varepsilon^{(1)}_k \right)^\frac{\log 2}{p \log x} \cdot \left( \frac{1}{1 - \frac{2}{x^p}} \right)^\frac{\log 2}{p \log x}  \cdot  \nu (I_l^k).
\end{split}
\end{equation*}
  We can now use corollary~\ref{cor:the_cor} to conclude that
\begin{equation*}
\begin{split}
\left( \left(1 - \varepsilon_k^{(0)}\right)^\frac{\log 2}{p \log x} - \left( \left(1 + \varepsilon_k^{(1)}\right)^\frac{\log 2}{p \log x} -\left(1 - \varepsilon_k^{(0)}\right)^\frac{\log 2}{p \log x} \right) \right) \cdot \left( \frac{1}{1 - \frac{2}{x^p}}\right) ^\frac{\log 2}{p \log x}  
\\ \leq
m_\frac{\log 2}{p \log x} (C_{p,x}) 
\end{split}
\end{equation*}
  and 
\begin{equation*}
m_\frac{\log 2}{p \log x} (C_{p,x}) \leq 
\left(1 + \varepsilon_k^{(1)}\right)^\frac{\log 2}{p \log x} \cdot \left( \frac{1}{1 - \frac{2}{x^p}}\right) ^\frac{\log 2}{p \log x} .
\end{equation*}
  By letting \( k \) tend to infinity in the two previous equations, we get
\begin{equation*}
m_\frac{\log 2}{p \log x} (C_{p,x}) =
\left( \frac{1}{1 - \frac{2}{x^p}}\right) ^\frac{\log 2}{p \log x}.
\end{equation*}
\end{example}

\begin{example}\label{ex:3}
  Our main theorem, Theorem~\ref{thm:general-1-thm}, which we used indirectly when calculating the Hausdorff measure of the Cantor sets in the previous two examples, can with small modifications be used also to calculate the measures of the sets in the third and last family of Cantor sets mentioned in \cite{pCantorset_measure}, namely the Cantor sets \( C_{(p)}^{(n)} \sim \{ G_l^k \} \), where \( |G_l^k|  = \frac{1}{(2^k + l)^p} \) and \(p>1\) as in the first example, but where \( n -1\) open intervals are removed from each remaining interval in each step of the construction of the Cantor set instead of one. Small adjustments to Theorem~\ref{thm:general-1-thm} and its proof and similar calculations as  in examples \ref{ex:1} and \ref{ex:2}, although omitted here, give
\begin{equation*}
m_{1/p} (C_{(p)}^{(n)}) = \frac{n \log n}{(n^p-n)^{1/p}} \cdot \frac{1}{n-1}.
\end{equation*} 
\end{example}

\paragraph{Acknowledgement.}

The author expresses her gratitude to her advisor, Maria Roginskaya, for her valuable suggestions and for reading the manuscript.


\begin{thebibliography}{99}


\bibitem{pCantorset_measure}
C. Cabrielli, U. M. Molter, F. Mendevil, V. Paulauskas and R. Shonkwiler,
\textit{Hausdorff measure of $p$-{C}antor sets}, 
Real Anal. Exchange \textbf{30(2)}, (2004) 413--434.

\bibitem{classics}
G. A. Edgar,
\textit{Classics on Fractals},
{Addison-Weasly Publishing Company}, (1993).

\bibitem{falconer2} 
K. J. Falconer
\textit{Fractal geometry - Mathematical Foundations and Applications}, 
Wiley, (2003).

\bibitem{smooth_cantor_sets}
I. Garcia,
\textit{A family of smooth {C}antor sets},
Ann. Acad. Sci. Fenn. Math. \textbf{36},
(2004) 21-45.

\bibitem{mattila} 
P. Mattila, 
\textit{Geometry of Sets and Measures in {E}uclidean Spaces},
Cambridge studies in advanced mathematics, \textbf{44}, Cambridge University Press, (1995).

\bibitem{malin}
M. Pal\"o,
\textit{Multidimensional {H}ausdorff measures on {C}antor sets},
Master Thesis at Gothenburg University, (2013).

\bibitem{hausdorffmeasures}
C. A. Rogers,
\textit{Hausdorff Measures,} Cambridge University Press, (1998).








 \end{thebibliography}
\end{document}